\documentclass[12pt]{amsart}
\usepackage{amsmath,amssymb,latexsym,esint,cite,mathrsfs}
\usepackage{verbatim,wasysym}
\usepackage[left=2.6cm,right=2.6cm,top=3cm,bottom=3cm]{geometry}
\makeatletter \@addtoreset{equation}{section} \makeatother

\usepackage{graphicx}
\usepackage{subfigure}
\usepackage{graphicx,epic,eepic}
\usepackage{makecell}
\usepackage{tikz}  
\usepackage[colorlinks=true,urlcolor=blue,
citecolor=red,linkcolor=blue,linktocpage,pdfpagelabels,
bookmarksnumbered,bookmarksopen]{hyperref}
\numberwithin{equation}{section}
\newtheorem{theorem}{Theorem}[section]
\newtheorem{lemma}[theorem]{Lemma}

\newtheorem{proposition}[theorem]{Proposition}
\newtheorem{corollary}[theorem]{Corollary}
\newtheorem{remark}[theorem]{Remark}
\numberwithin{equation}{section}

\allowdisplaybreaks

\begin{document}

\title[Symmetry breaking of (CKN) inequality]
{Symmetry breaking of extremals for the high order Caffarelli-Kohn-Nirenberg type inequalities: the singular case}

\author[S. Deng]{Shengbing Deng$^{\ast}$}
\address{\noindent Shengbing Deng
\newline
School of Mathematics and Statistics, Southwest University,
Chongqing 400715, People's Republic of China}\email{shbdeng@swu.edu.cn}

\author[X. Tian]{Xingliang Tian}
\address{\noindent Xingliang Tian  \newline
School of Mathematics and Statistics, Southwest University,
Chongqing 400715, People's Republic of China}\email{xltian@email.swu.edu.cn}

\thanks{$^{\ast}$ Corresponding author}



\allowdisplaybreaks

\begin{abstract}
{\tiny
Uniqueness of radial solutions for the following system are obtained
    \begin{eqnarray}\label{Pwhs0}
    \left\{ \arraycolsep=1.5pt
       \begin{array}{ll}
        -\mathrm{div}(|x|^\alpha\nabla u)=|x|^{\beta}v,\quad &\ \ \ \mbox{in}\ \mathbb{R}^N\setminus\{0\},\\[2mm]
        -\mathrm{div}(|x|^{\alpha}\nabla v)=|x|^\gamma|u|^{2^{**}_{\alpha,\beta}-2}u,\qquad &\ \ \ \mbox{in}\ \mathbb{R}^N\setminus\{0\},
        \end{array}
    \right.
    \end{eqnarray}
which relates to a second-order Caffarelli-Kohn-Nirenberg type inequality
    \begin{equation}\label{nsckn}
    \int_{\mathbb{R}^N}|x|^{-\beta}|\mathrm{div} (|x|^{\alpha}\nabla u)|^2 \mathrm{d}x
    \geq \mathcal{S}\left(\int_{\mathbb{R}^N}|x|^{\gamma}
    |u|^{2^{**}_{\alpha,\beta}} \mathrm{d}x\right)^{\frac{2}{2^{**}_{\alpha,\beta}}}, \quad \mbox{for all}\quad u\in C^\infty_0(\mathbb{R}^N\setminus\{0\}),
    \end{equation}
    for some $\mathcal{S}=\mathcal{S}(N,\alpha,\beta)>0$, where $N\geq 5$, $\alpha>2-N$, $\frac{N-4}{N-2}\alpha-4 \leq \beta\leq\alpha -2$
    and
    \begin{align*}
    2^{**}_{\alpha,\beta}:=\frac{2(N+\gamma)}{N+2\alpha-\beta-4}
    \quad \mbox{with}\quad (N+\beta)(N+\gamma)=(N+2\alpha-\beta-4)^2.
    \end{align*}
    A crucial element is that the functional $\int_{\mathbb{R}^N}|x|^{-\beta}|\mathrm{div} (|x|^{\alpha}\nabla u)|^2 \mathrm{d}x$ is equivalent to $\int_{\mathbb{R}^N}|x|^{2\alpha-\beta}|\Delta u|^2 \mathrm{d}x$.
    Firstly, we obtain a symmetry result (with partial translation invariant) when $\alpha=0$ and $\beta=-4$, then existence and non-existence of extremal functions for the best constant $\mathcal{S}$ in \eqref{nsckn} under different conditions are completely given. Moreover, by a result of linearized problem related to radial solution of \eqref{Pwhs0}, we obtain a symmetry breaking conclusion: when $\alpha>0$ and $\frac{N-4}{N-2}\alpha-4<\beta<\beta_{\mathrm{FS}}(\alpha)$ where $\beta_{\mathrm{FS}}(\alpha):=
    N+2\alpha-4-\sqrt{(N-2+\alpha)^2+4(N-1)}$, the extremal functions for $\mathcal{S}$ are nonradial. Finally, we give a partial symmetry result when $\beta=\frac{N-4}{N-2}\alpha-4$ and $2-N<\alpha<0$, and we also study the stability of extremal functions.
    This extends the works of Catrina-Wang [Comm. Pure Appl. Math., 2001], Felli-Schneider [J. Differ. Equ., 2003] and Dolbeault-Esteban-Loss [Invent. Math., 2016] to the second-order case.
    }
\end{abstract}

\vspace{3mm}

\maketitle

\section{{\bfseries Introduction}}\label{sectir}

\subsection{Motivation}\label{subsectmot}
Let us recall the famous so-called first order  Caffarelli-Kohn-Nirenberg (we write (CKN) for short) inequality which was first introduced in 1984 by Caffarelli, Kohn and Nirenberg in their celebrated work \cite{CKN84}.  Here, we only focus on the case without interpolation term, that is,
    \begin{equation}\label{cknwit}
    \left(\int_{\mathbb{R}^N}|x|^{-br}|u|^r \mathrm{d}x\right)^{\frac{2}{r}}
    \leq C_{\mathrm{CKN}}\int_{\mathbb{R}^N}|x|^{-2a}|\nabla u|^2 \mathrm{d}x, \quad \mbox{for all}\quad u\in C^\infty_0(\mathbb{R}^N),
    \end{equation}
    for some constant $C_{\mathrm{CKN}}>0$, where
    \begin{equation*}
    -\infty<a<a_c:=\frac{N-2}{2},\quad a\leq b\leq a+1,\quad r=\frac{2N}{N-2(1+a-b)}.
    \end{equation*}
    A natural question is whether the best constant $C_{\mathrm{CKN}}$ could be achieved or not? Moreover, if $C_{\mathrm{CKN}}$ is achieved, are the extremal functions radial symmetry? In fact, the weights $|x|^{-2a}$ and $|x|^{-br}$ have deep influences in many aspects about this inequality, for example, achievable of best constant and symmetry of minimizers.

    When $0\leq a<a_c$ and $N\geq 3$, Chou and Chu \cite{CC93} obtained $C_{\mathrm{CKN}}$ is achieved by explicit radial function using the moving planes and symmetrization methods, see also \cite[Lemma 2.1]{DELT09} for a simpler proof.
    When $b=a+1$ or $b=a<0$, Catrina and Wang \cite{CW01} proved that $C_{\mathrm{CKN}}$ is not achieved and for other cases it is always achieved. Furthermore, when $N\geq 3$, $a<0$ and $a< b<b_{\mathrm{FS}}(a)$, where
    \[
    b_{\mathrm{FS}}(a):=\frac{N(a_c-a)}{2\sqrt{(a_c-a)^2+N-1}}
    +a-a_c,
    \]
    Felli and Schneider \cite{FS03} proved the extremal function is nonradial by using the method of restricting it in the radial space and classifying linearized problem. Therefore, the function $b_{\mathrm{FS}}$ is usually called as {\em Felli-Schneider curve}. Dolbeault et al. \cite{DET08} also obtained the same symmetry breaking conclusion when $N=2$. Finally, in a celebrated paper, Dolbeault, Esteban and Loss \cite{DEL16} proved an optimal rigidity result by using the so-called {\em carr\'{e} du champ} method that when $a<0$ and $b_{\mathrm{FS}}(a)\leq b<a+1$, the extremal function is symmetry. See the previous results shown as in Figure \ref{F1}.

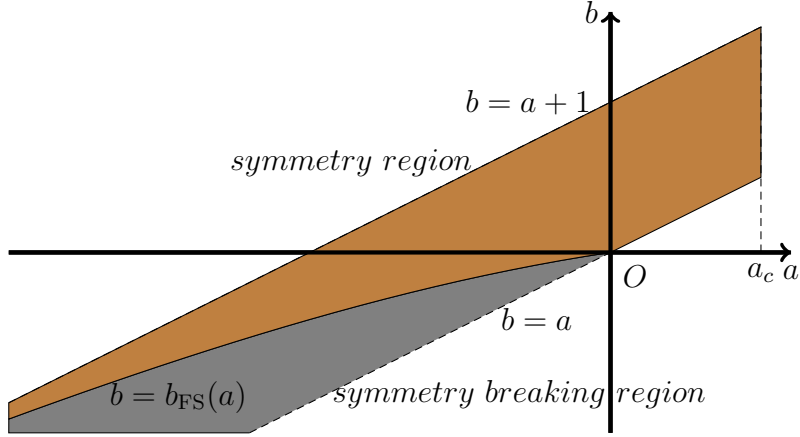
\begin{figure}
\begin{tikzpicture}[scale=4]
		\draw[->,ultra thick](-2,0)--(0,0)node[below right]{$O$}--(0.6,0)node[below]{$a$};
		\draw[->,ultra thick](0,-0.6)--(0,0.8)node[left]{$b$};
        \draw[fill=gray,domain=0:-2]plot(\x,{4*(1-0.47*\x)/
        (2*((1-0.6*\x)^2+3)^0.5)
        +0.47*\x -1})--(-2,-0.6)--(-1.2,-0.6);
        \draw[fill=brown,domain=0:-2]plot(\x,{4*(1-0.47*\x)/
        (2*((1-0.6*\x)^2+3)^0.5)
        +0.47*\x -1})--(-2,-0.5)--(0.5,0.75)--(0.5,0.25)--(0,0);
        \draw[densely dashed](0.5,0)node[below]{$a_c$}--(0.5,0.75);
        \draw[densely dashed](-2,-0.5)--(0.5,0.75);
        \draw[densely dashed](-1.2,-0.6)--(0,0);
        \draw[-,ultra thick](0,-0.6)--(0,0.8);
        \draw[-,ultra thick](-2,0)--(0.6,0);

		\node[left] at(-0.03,0.5){$b=a+1$};
        \node[right] at (-0.4,-0.22){$b=a$};
        \node[right] at (-1.7,-0.48){$b=b_{\mathrm{FS}}(a)$};
        \node[right] at (-0.96,-0.47){$symmetry\ breaking\ region$};
        \node[right] at (-1.3,0.3){$symmetry\ region$};
\end{tikzpicture}
\caption{\small The {\em Felli-Schneider region}, or symmetry breaking region, appears in dark grey and is defined by $a<0$ and $a<b<b_{\mathrm{FS}}(a)$. And symmetry holds in the brown region defined by $a<0$ and $ b_{\mathrm{FS}}(a)\leq b<a+1$, also $0\leq a<a_c$ and $a\leq b<a+1$.}
\label{F1}
\end{figure}

    In 1986, C.-S. Lin \cite{Li86} extended the (CKN) inequality of \cite{CKN84} to the higher-order case, here we only mention the case without interpolation term:
    \vskip0.25cm

    \noindent{\bf Theorem~A.} \cite{Li86} {\it  There exists a constant $C>0$ such that
    \begin{equation}\label{cknh}
    \||x|^{-a} D^m u\|_{L^p}\geq C \||x|^{-b} D^j u\|_{L^r},
    \quad \mbox{for all}\quad u\in C^\infty_0(\mathbb{R}^N),
    \end{equation}
    where $j\geq 0$, $m>0$ are integers, and
    \begin{align*}
    \frac{1}{p}-\frac{a}{N}>0,\quad \frac{1}{r}-\frac{b}{N}>0, \quad
    a\leq b\leq a+m-j, \quad
    \frac{1}{r}-\frac{b+j}{N}=\frac{1}{p}-\frac{a+m}{N}.
    \end{align*}
    Here
    \begin{eqnarray*}
    D^s u:=\left\{ \arraycolsep=1.5pt
       \begin{array}{ll}
        \nabla (\Delta^{\frac{s-1}{2}}u),\quad \mbox{if}\ s\ \mbox{is odd};\\[2mm]
        \Delta^{\frac{s}{2}}u,\quad \mbox{if}\ s\ \mbox{is even}.
        \end{array}
    \right.
    \end{eqnarray*}
    In particular, when $j=0$, $m=2$ and $p=2$, it holds
    \begin{equation}\label{ckn2Y}
    \int_{\mathbb{R}^N}|x|^{-2a}|\Delta u|^2 \mathrm{d}x
    \geq C\left(\int_{\mathbb{R}^N}|x|^{-rb}|u|^r \mathrm{d}x\right)^{\frac{2}{r}},
    \quad \mbox{for all}\quad u\in C^\infty_0(\mathbb{R}^N),
    \end{equation}
    where $-\infty<a<\frac{N-4}{2}$, $a\leq b\leq a+2$, $r=\frac{2N}{N-2(2+a-b)}$.
    }

    Therefore, same as the first-order case, it is natural to ask whether the best constant could be achieved or not for the second-order case? Moreover, if it is achieved, are the extremal functions radial symmetry? There are some partial results about these problems for the second-order case \eqref{ckn2Y}: Szulkin and Waliullah \cite{SW12} proved when $a=b>0$ the best constant is achieved, furthermore, Caldiroli and Musina
    \cite{CM11} proved when $a<b<a+2$ or $a=b$ with some other assumptions, the best constant is always achieved, see also \cite{MS14}.
    Caldiroli and Cora \cite{CC16} obtained a partial symmetry breaking result when the parameter of pure Rellich term is sufficiently large, we refer to \cite{CM11} in cones, and also \cite{Ya21} with Hardy and Rellich terms. Dong \cite{Do18} obtained the existence of extremal functions of higher-order (CKN) inequality \eqref{cknh} and found the sharp constants under some suitable assumptions. However, as mentioned previous, there are no optimal results about symmetry or symmetry breaking phenomenon.

    For the classical second-order Sobolev inequality
    \begin{equation}\label{cssio}
    \int_{\mathbb{R}^N}|\Delta u|^2 \mathrm{d}x \geq \mathcal{S}_0\left(\int_{\mathbb{R}^N}|u|^{2^{**}} \mathrm{d}x\right)^{\frac{2}{2^{**}}}, \quad \mbox{for all}\quad  u\in \mathcal{D}^{2,2}_0(\mathbb{R}^{N}),
    \end{equation}
    where $N\geq 5$, $2^{**}:=\frac{2N}{N-4}$, $\mathcal{D}^{2,2}_0(\mathbb{R}^{N})$ denotes the completion of $C^\infty_0(\mathbb{R}^{N})$ with respect to the norm $\|u\|_{\mathcal{D}^{2,2}_0(\mathbb{R}^{N})}
    =\left(\int_{\mathbb{R}^{N}}|\Delta u|^2 \mathrm{d}x\right)^{1/2}$ and
    \begin{equation}\label{cssi}
    \mathcal{S}_0=\pi^2N(N-4)(N^2-4)
    \left\{\frac{\Gamma(N/2)}{\Gamma(N)}\right\}^{4/N}
    \quad\mbox{with}\ \Gamma(\gamma):=\int^{+\infty}_0 t^{\gamma-1} e^{-t}\mathrm{d}t\ \mbox{for}\ \gamma>0,
    \end{equation}
    is the best constant given as in \cite{Va93}, it is well known that $\mathcal{S}_0$ is achieved only by $(1+|x|^2)^{-\frac{N-4}{2}}$ (up to multiplications, translations and scalings), see  \cite{Li85-1,EFJ90}. Furthermore, for the weighted Sobolev inequality (or the so-called Rellich-Sobolev inequality)
    \begin{equation*}
    \int_{\mathbb{R}^N}|\Delta u|^2 \mathrm{d}x \geq \mathcal{S}_\alpha\left(\int_{\mathbb{R}^N}
    \frac{|u|^{2^{**}_{\alpha}}}{|x|^{\alpha}} \mathrm{d}x\right)^{\frac{2}{2^{**}_{\alpha}}}, \quad \mbox{for all}\quad  u\in \mathcal{D}^{2,2}_0(\mathbb{R}^{N}),
    \end{equation*}
    where $N\geq 5$, $0<\alpha<4$ and $2^{**}_{\alpha}=\frac{2(N-\alpha)}{N-4}$. It is well known that the best constant $\mathcal{S}_\alpha$ is also achieved by radial function (see the classical result of Lions \cite{Li85-2}), however the explicit form of minimizer is not known yet only for its asymptotic behavior, we refer to \cite{JL14} for details, and also \cite{KX17} with pure Rellich potentials.
    Recently, in \cite{DT23-jfa}, we have considered the following second-order (CKN) type inequality:
    \begin{equation*}
    \int_{\mathbb{R}^N}|x|^{\alpha}|\Delta u|^2 \mathrm{d}x \geq S_1^{rad}(N,\alpha)\left(\int_{\mathbb{R}^N}|x|^{-\alpha}|u|^{p^*_{\alpha}} \mathrm{d}x\right)^{\frac{2}{p^*_{\alpha}}}, \quad u\in C^\infty_0(\mathbb{R}^N)\ \mbox{and}\ u\ \mbox{is radial},
    \end{equation*}
    where $N\geq 3$, $4-N<\alpha<2$, $p^*_{\alpha}=\frac{2(N-\alpha)}{N-4+\alpha}$.
    As in \cite{FS03} which deals with the first-order (CKN) inequality, we restrict it in radial space, then consider the best constant $S_1^{rad}(N,\alpha)$ and its minimizer $V$ by using the change of variable $v(s)=u(r)$ where $r=s^{2/(2-\alpha)}$ which transfers it into the standard second-order Sobolev inequality
    \begin{equation*}
    \int^\infty_0\left[v''(s)+\frac{M-1}{s}v'(s)\right]^2 s^{M-1}\mathrm{d}s
    \geq \mathcal{B}(M)\left(\int^\infty_0|v(s)|^{\frac{2M}{M-4}}
    s^{M-1}\mathrm{d}s\right)^{\frac{M-4}{M}},
    \end{equation*}
    for some suitable $\mathcal{B}(M)>0$, where $M=\frac{2N-2\alpha}{2-\alpha}>4$,
    and also classify the solutions of related linearized problem:
    \begin{equation*}
    \Delta(|x|^{\alpha}\Delta v)=(p^*_{\alpha}-1)|x|^{-\alpha} V^{p^*_{\alpha}-2}v \quad \mbox{in}\ \mathbb{R}^N, \quad v\in C^\infty_0(\mathbb{R}^N).
    \end{equation*}
    We have showed that if $\alpha$ is a negative even integer, there exist new solutions which ``replace'' the ones due to the translations invariance. Moreover, in \cite{DGT23-jde}, we have also considered the singular case, by using the change of variable $v(s)=r^{2-N}u(r)$ where $r=s^{\frac{2}{2-\alpha}}$,
    \begin{equation*}
    \int_{\mathbb{R}^N}|x|^{\alpha}|\Delta u|^2 \mathrm{d}x \geq S_2^{rad}(N,\alpha)\left(\int_{\mathbb{R}^N}|x|^{l}|u|^{q^*_{\alpha}} \mathrm{d}x\right)^{\frac{2}{q^*_{\alpha}}}, \quad u\in C^\infty_0(\mathbb{R}^N\setminus\{0\})\ \mbox{and}\ u\ \mbox{is radial},
    \end{equation*}
    where $N\geq 3$, $2<\alpha<N$, $l=\frac{4(\alpha-2)(N-2)}{N-\alpha}-\alpha$ and $q^*_{\alpha}=\frac{2(N+l)}{N-4+\alpha}$. Furthermore, in \cite{DT24} we considered a new type weighted fourth-order equation, and established a new second-order Caffarelli-Kohn-Nirenberg type inequality without restricting in radial space.

    However, it seems difficult to obtain the symmetry breaking result as the Felli-Schneider type \cite{FS03}. Based on the work of Lin \cite{Li86}, very recently in \cite{DT23-f}, we establish a new second-order Caffarelli-Kohn-Nirenberg type inequality
    \begin{equation*}
    \int_{\mathbb{R}^N}|x|^{-\beta}|\mathrm{div} (|x|^{\alpha}\nabla u)|^2 \mathrm{d}x
    \geq S_{\alpha,\beta}\left(\int_{\mathbb{R}^N}|x|^{\beta}
    |u|^{p^*_{\alpha,\beta}} \mathrm{d}x\right)^{\frac{2}{p^*_{\alpha,\beta}}}, \quad \mbox{for all}\quad u\in C^\infty_0(\mathbb{R}^N),
    \end{equation*}
    for some constant $S_{\alpha,\beta}>0$, where
    \begin{align*}
    N\geq 5,\quad \alpha>2-N,\quad \alpha-2<\beta\leq \frac{N}{N-2}\alpha,\quad p^*_{\alpha,\beta}=\frac{2(N+\beta)}{N-4+2\alpha-\beta}.
    \end{align*}
    By using the method as in \cite{DT23-jfa}, we obtain a symmetry breaking conclusion: when $\alpha>0$ and $-N+\sqrt{N^2+\alpha^2+2(N-2)\alpha}<\beta< \frac{N}{N-2}\alpha$, then the extremal function for the best constant $S_{\alpha,\beta}$, if it exists, is nonradial. Therefore, it is natural to ask whether we can establish symmetry breaking conclusion as \cite{DT23-f} for singular case by using the method of \cite{DGT23-jde} or not? We will give an affirmative answer.

\subsection{Problem setup and main results}\label{subsectmr}

    In present paper, we also do not directly deal with the high order (CKN) inequality \eqref{ckn2Y} established by Lin \cite{Li86}, but we will establish a second-order (CKN) type inequality, namely
    \begin{equation}\label{ckn2n}
    \int_{\mathbb{R}^N}|x|^{-\beta}|\mathrm{div} (|x|^{\alpha}\nabla u)|^2 \mathrm{d}x
    \geq \mathcal{S}\left(\int_{\mathbb{R}^N}
    |x|^{\gamma}|u|^{2^{**}_{\alpha,\beta}} \mathrm{d}x\right)^{\frac{2}{2^{**}_{\alpha,\beta}}}, \quad \mbox{for all}\quad u\in \mathcal{D}^{2,2}_{\alpha,\beta}(\mathbb{R}^N),
    \end{equation}
    for some constant $\mathcal{S}=\mathcal{S}(N,\alpha,\beta)>0$, where 
    \begin{align}\label{cknc}
    N\geq 5,\quad \alpha>2-N,\quad \frac{N-4}{N-2}\alpha-4 \leq \beta\leq \alpha -2,
    \end{align}
    and
    \begin{align}\label{cknc1}
    2^{**}_{\alpha,\beta}=\frac{2(N+\gamma)}{N+2\alpha-\beta-4},
    \quad \mbox{with}\quad (N+\beta)(N+\gamma)=(N+2\alpha-\beta-4)^2.
    \end{align}
    Here we denote $\mathcal{D}^{2,2}_{\alpha,\beta}(\mathbb{R}^N)$ as the completion of $C^\infty_0(\mathbb{R}^N\setminus\{0\})$ with respect to the norm
    \begin{equation}\label{defd22i}
    \|u\|_{\mathcal{D}^{2,2}_{\alpha,\beta}(\mathbb{R}^N)}
    =\left(\int_{\mathbb{R}^N}|x|^{-\beta}|\mathrm{div} (|x|^{\alpha}\nabla u)|^2 \mathrm{d}x\right)^{\frac{1}{2}}.
    \end{equation}
    Note that \eqref{cknc} and \eqref{cknc1} imply $N+2\alpha-\beta-4>0$, $N+\beta>0$ and $2\leq2^{**}_{\alpha,\beta}\leq\frac{2N}{N-4}$.

    Taking the same arguments as those in \cite[Section 2]{GG22} (which deals with it in bounded domain, but it also holds in the whole space), then we deduce that for all $u\in C^\infty_0(\mathbb{R}^N\setminus\{0\})$,
    \[
    \int_{\mathbb{R}^N} |x|^{-\beta}|\mathrm{div}(|x|^\alpha\nabla u)|^2 \mathrm{d}x\quad \mbox{is equivalent to}\quad \int_{\mathbb{R}^N} |x|^{2\alpha-\beta}|\Delta u|^2 \mathrm{d}x.
    \]
    For the convenience of readers, we give the proof in Proposition \ref{propneq}.
    Note that \eqref{cknc}-\eqref{cknc1} imply
    \[
    -\frac{2\alpha-\beta}{2}<\frac{N-4}{2}\quad\mbox{and}\quad
    -\frac{2\alpha-\beta}{2}\leq
    -\frac{\gamma}{2^{**}_{\alpha,\beta}}\leq -\frac{2\alpha-\beta}{2}+2,
    \]
    then in \eqref{ckn2Y}, let $a=-\frac{2\alpha-\beta}{2}$ and $b=-\frac{\gamma}{2^{**}_{\alpha,\beta}}$, we directly obtain the new type second-order (CKN) inequality \eqref{ckn2n} which is equivalent to the classical one \eqref{ckn2Y}.
    Now, let us rewrite \eqref{ckn2n} as
    \begin{equation}\label{ckns}
    \mathcal{S}:=\inf_{u\in \mathcal{D}^{2,2}_{\alpha,\beta}(\mathbb{R}^N)\setminus\{0\}}
    \frac{\int_{\mathbb{R}^N}|x|^{-\beta}|\mathrm{div} (|x|^{\alpha}\nabla u)|^2 \mathrm{d}x}
    {\left(\int_{\mathbb{R}^N}|x|^{\gamma}|u|^{2^{**}_{\alpha,\beta}} \mathrm{d}x\right)^{\frac{2}{2^{**}_{\alpha,\beta}}}}>0.
    \end{equation}
    We are interested in whether the extremal functions (if exist) of the best constant $\mathcal{S}$ are symmetry or not.

    Firstly, we give a symmetry result when $\alpha=0$ and $\beta=-4$ which is crucial for the existence and non-existence of extremal functions for critical case $2^{**}_{\alpha,\beta}=2^{**}$.

    \begin{theorem}\label{thmsr}
    Let $N\geq 5$. It holds that
    \begin{equation}\label{ckn2ns}
    \int_{\mathbb{R}^N}|x|^{4}|\Delta u|^2 \mathrm{d}x
    \geq \mathcal{S}_0\left(\int_{\mathbb{R}^N}
    |x|^{2\cdot2^{**}}|u|^{2^{**}} \mathrm{d}x\right)^{\frac{2}{2^{**}}}, \quad \mbox{for all}\quad u\in \mathcal{D}^{2,2}_{0,-4}(\mathbb{R}^N),
    \end{equation}
    where $\mathcal{S}_0$ is the classical best constant of second-order Sobolev inequality given as in \eqref{cssi}. Furthermore, equality holds if and only if
    \begin{equation}\label{pbcm0-4}
    u(x)=u_{\lambda,x_0}(x)
    =\frac{C\lambda^{\frac{N-4}{2}}}
    {|x|^{2}(1+\lambda^{2}|x-x_0|^{2})
    ^{\frac{N-4}{2}}},
    \end{equation}
    for all $C\in\mathbb{R}$, $\lambda>0$ and $x_0\in\mathbb{R}^N$.
    \end{theorem}

    \begin{remark}\label{rmsr}
    It is worth noticing that although there is no translation invariance in \eqref{ckn2ns}, we obtain partial translation invariance result.
    \end{remark}

    \begin{theorem}\label{thmene}
    {\rm (Existence and Non-existence of Extremal Functions)}
    Assume that \eqref{cknc}-\eqref{cknc1} hold.
     \begin{itemize}
    \item[$(i)$]
    For $\beta=\alpha -2$, then we can choose
    \begin{align*}
    \mathcal{S}=\left(\frac{N(N-4)}{4}\right)^2
    +2(Q-N+2)\left(\frac{N-4}{2}\right)^2
    +Q^2,
    \end{align*}
    where $Q=\frac{2+\alpha}{2}
    \left(N-2+\alpha-\frac{2+\alpha}{2}\right)$, such that
    \begin{equation*}
    \int_{\mathbb{R}^N}|x|^{2-\alpha}|\mathrm{div} (|x|^{\alpha}\nabla u)|^2 \mathrm{d}x
    \geq \mathcal{S}\int_{\mathbb{R}^N}
    |x|^{\alpha-2}|u|^{2} \mathrm{d}x, \quad \mbox{for all}\quad u\in C^\infty_0(\mathbb{R}^N\setminus\{0\}),
    \end{equation*}
    furthermore, $\mathcal{S}$ is sharp and not achieved for nonzero $u$. Particularly, we obtain a new Rellich type inequality
    \begin{align*}
    \int_{\mathbb{R}^N}|x|^{4}|\mathrm{div} (|x|^{-2}\nabla u)|^2 \mathrm{d}x
    \geq \left(\frac{N-4}{2}\right)^4\int_{\mathbb{R}^N}
    \frac{|u|^2}{|x|^4}\mathrm{d}x, \quad \mbox{for all}\quad u\in C^\infty_0(\mathbb{R}^N\setminus\{0\}),
    \end{align*}
    and the constant $\left(\frac{N-4}{2}\right)^4$ is sharp  and not achieved for nonzero $u$.
    \item[$(ii)$]
    For $\beta= \frac{N-4}{N-2}\alpha-4$ with $\alpha>0$, then $\mathcal{S}=\mathcal{S}_0$ and it is not achieved.
    \item[$(iii)$]
    For $\frac{N-4}{N-2}\alpha-4 < \beta<\alpha -2$, then $\mathcal{S}$ is achieved.
    \item[$(iv)$]
    For $\beta=\frac{N-4}{N-2}\alpha-4$ with $\alpha< 0$, then $\mathcal{S}<\mathcal{S}_0$ and it is always achieved.
    \end{itemize}
    \end{theorem}

    Then based on the above results, when $\mathcal{S}$ is achieved, it is natural to ask whether the extremal functions are symmetry or symmetry breaking? We will give an  affirmative answer to symmetry breaking under some suitable conditions.

    Following the work of Felli and Schneider \cite{FS03}, then let us consider the radial case. The Euler-Lagrange equation of \eqref{ckn2n}, up to scaling, is given by
    \begin{equation}\label{Pwh}
    \mathrm{div}(|x|^{\alpha}\nabla(|x|^{-\beta}
    \mathrm{div}(|x|^\alpha\nabla u)))=|x|^\gamma|u|^{2^{**}_{\alpha,\beta}-2}u \quad \mbox{in}\  \mathbb{R}^N\setminus\{0\},\quad u\in \mathcal{D}^{2,2}_{\alpha,\beta}(\mathbb{R}^N).
    \end{equation}
    Note that \eqref{Pwh} is equivalent to the system
    \begin{eqnarray*}
    \left\{ \arraycolsep=1.5pt
       \begin{array}{ll}
        -\mathrm{div}(|x|^\alpha\nabla u)=|x|^{\beta}v&\quad \mbox{in}\  \mathbb{R}^N\setminus\{0\},\\[2mm]
        -\mathrm{div}(|x|^{\alpha}\nabla v)=|x|^\gamma|u|^{2^{**}_{\alpha,\beta}-2}u&\quad \mbox{in}\  \mathbb{R}^N\setminus\{0\},
        \end{array}
    \right.
    \end{eqnarray*}
    which is a special weighted Lane-Emden system.

    \begin{theorem}\label{thmpwh}
    Assume that \eqref{cknc}-\eqref{cknc1} hold with $\beta\neq \alpha-2$. Then problem \eqref{Pwh} has a unique (up to scalings and change of sign)
    radial solution of the form $\pm U_{\lambda}$ for some $\lambda>0$, where $U_{\lambda}(x)= \lambda^{\frac{N+2\alpha-\beta-4}{2}}U(\lambda x)$ with
    \begin{align}\label{defula}
    U(x)=\frac{C_{N,\alpha,\beta}}
    {|x|^{\alpha-\beta-2}(1+|x|^{\alpha-\beta-2})
    ^{\frac{N+\beta}{\alpha-\beta-2}}}.
    \end{align}
    Here $C_{N,\alpha,\beta}=\left[(N+\beta)(N+\alpha-2)
    (N+2\alpha-\beta-4)(N+3\alpha-2\beta-6)\right]
    ^{\frac{N+\beta}{4(\alpha -\beta-2)}}$.
    \end{theorem}

    As a direct consequence of Theorem \ref{thmpwh}, we obtain

    \begin{corollary}\label{thmPbcb}
    Assume that \eqref{cknc}-\eqref{cknc1} hold with $\beta\neq \alpha-2$. Let us define the best constant in the radial class to be
    \begin{equation}\label{Ppbcm}
    \mathcal{S}_r:=\inf_{\thead{u\in \mathcal{D}^{2,2}_{\alpha,\beta}(\mathbb{R}^N)\setminus\{0\} \\ u(x)=u(|x|)}
    }
    \frac{\int_{\mathbb{R}^N}|x|^{-\beta}|\mathrm{div} (|x|^{\alpha}\nabla u)|^2 \mathrm{d}x}
    {\left(\int_{\mathbb{R}^N}|x|^{\gamma}|u|^{2^{**}_{\alpha,\beta}} \mathrm{d}x\right)^{\frac{2}{2^{**}_{\alpha,\beta}}}},
    \end{equation}
    then it's explicit form is
    \begin{equation}\label{defsr}
    \mathcal{S}_r
    =\left(\frac{2}{\alpha-\beta-2}\right)
    ^{\frac{2(\alpha-\beta-2)}{N+2\alpha-\beta-4}-4}
    \left(\frac{2\pi^{\frac{N}{2}}}{\Gamma(\frac{N}{2})}\right)
    ^{\frac{2(\alpha-\beta-2)}{N+2\alpha-\beta-4}}
    \mathcal{B}\left(\frac{2(N+2\alpha-\beta-4)}{\alpha-\beta-2}\right),
    \end{equation}
    where $\mathcal{B}(M)=(M-4)(M-2)M(M+2)
    \left[\Gamma^2(\frac{M}{2})/(2\Gamma(M))\right]^{\frac{4}{M}}$.
    Moreover the extremal functions of $\mathcal{S}_r$ are given as $CU_{\lambda}(x)$
    for all $C\in\mathbb{R}\setminus\{0\}$ and $\lambda>0$.
    \end{corollary}

    Then we concern the linearized problem related to \eqref{Pwh} at the function $U$. This leads to study the problem
    \begin{small}\begin{equation}\label{Pwhl}
    \mathrm{div}(|x|^{\alpha}\nabla(|x|^{-\beta}
    \mathrm{div}(|x|^\alpha\nabla u)))=(2^{**}_{\alpha,\beta}-1)|x|^\gamma U^{2^{**}_{\alpha,\beta}-2}u \quad \mbox{in}\ \mathbb{R}^N\setminus\{0\},\quad v\in \mathcal{D}^{2,2}_{\alpha,\beta}(\mathbb{R}^N).
    \end{equation}\end{small}
    It is easy to verify that $\frac{N+2\alpha-\beta-4}{2}U+x\cdot \nabla U$ (which equals $\frac{\partial U_{\lambda}}{\partial \lambda}|_{\lambda=1}$) solves the linear equation \eqref{Pwhl}. We say that $U$ is non-degenerate if all solutions of \eqref{Pwh} result from the invariance (up to scalings) of \eqref{Pwhl}. The non-degeneracy of solutions for \eqref{Pwhl} is a key ingredient in analyzing the blow-up phenomena of solutions to various elliptic equations on bounded or unbounded domain in $\mathbb{R}^N$ whose asymptotic behavior is encoded in \eqref{defula}, and also for the stability of extremal functions for inequality.
    Therefore, it is quite natural to ask the following question:
    \begin{center}
    {\em is solution $U$ non-degenerate?}
    \end{center}
    Let us define a function
    \begin{align}\label{defbfs}
    \beta_{\mathrm{FS}}(\alpha):=
        N+2\alpha-4-\sqrt{(N-2+\alpha)^2+4(N-1)}.
    \end{align}
    We give an affirmative answer to non-degeneracy when $\beta_{\mathrm{FS}}(\alpha)<\beta<\alpha-2$ if $\alpha>0$ and $\frac{N-4}{N-2}\alpha-4\leq\beta<\alpha-2$ if $2-N<\alpha<0$, however when $\beta=\beta_{\mathrm{FS}}(\alpha)$ with $\alpha\geq 0$ there exist new solutions to the linearized problem that ``replace'' the ones due to the translations invariance. This can be stated as follows.

    \begin{theorem}\label{thmpwhl}
    Assume that \eqref{cknc}-\eqref{cknc1} hold, and $\beta_{\mathrm{FS}}(\alpha)\leq\beta<\alpha-2$ if $\alpha\geq0$, and $\frac{N-4}{N-2}\alpha-4\leq \beta<\alpha-2$ if $2-N<\alpha<0$. If $\beta=\beta_{\mathrm{FS}}(\alpha)$ then the space of solutions of (\ref{Pwhl}) has dimension $(1+N)$ and is spanned by
    \begin{equation}\label{defaezki}
    Z_{0}(x)=\frac{1-|x|^{2+\beta-\alpha}}
    {(1+|x|^{\alpha-\beta-2})^{\frac{N-2+\alpha}{\alpha-\beta-2}}},\quad Z_{i}(x)=\frac{|x|^{\frac{2+\beta-\alpha}{2}}}{(1+|x|^{\alpha-\beta-2})
    ^\frac{N-2+\alpha}{\alpha-\beta-2}}\cdot\frac{x_i}{|x|},\ i=1,\ldots,N.
    \end{equation}
    Otherwise, the space of solutions of \eqref{Pwhl} has only dimension one and is spanned by $Z_0\thicksim \frac{\partial U_{\lambda}}{\partial \lambda}|_{\lambda=1}$, and in this case
    we say the solution $U$ of equation \eqref{Pwh} is non-degenerate.
    \end{theorem}

    \begin{remark}\label{rme}
    It is worth noticing that $Z_i\not\sim \frac{\partial U}{\partial x_i}$ for every $i\in \{1,\ldots,N\}$. But in the case $\alpha=0$ and $\beta=-4$, it is easy to verify that  $Z_0\sim \frac{\partial u_{\lambda,0}(x)}{\partial \lambda}|_{\lambda=1}$ and $Z_i\sim \frac{\partial u_{1,x_0}(x)}{\partial x_{0,i}}|_{x_0=0}$ for all $i\in \{1,\ldots,N\}$, where $u_{\lambda,x_0}(x)$ is given as in \eqref{pbcm0-4} with some suitable $C>0$, thus this case can be seem as another form of classical non-degenerate result for pure critical fourth-order equation (see \cite{BWW03}).
    \end{remark}

    Now, we are ready to give the main result of this paper.

    \begin{theorem}\label{thmmr}
    Let $N\geq 5$, $\alpha>0$ and $\frac{N-4}{N-2}\alpha-4<\beta<\beta_{\mathrm{FS}}(\alpha)$. Then $\mathcal{S}<\mathcal{S}_r$, that is, the extremal function for the best constant $\mathcal{S}$ which is defined in \eqref{ckns}, is nonradial.
    \end{theorem}

    Comparing our proof to the ones in \cite{CW01,DEL16,FS03}, we also call $\beta_{\mathrm{FS}}(\alpha)$ which is given in \eqref{defbfs} as {\em Felli-Schneider curve}, and we give the following conjecture.

    \vskip0.25cm

    {\em {\bf Conjecture:} the extremal functions of $\mathcal{S}$ are radial symmetry either $2-N<\alpha\leq 0$ and $\frac{N-4}{N-2}\alpha-4\leq \beta< \alpha-2$ (except for $\alpha=0$ and $\beta=-4$), or $\alpha>0$ and $ \beta_{\mathrm{FS}}(\alpha)\leq \beta<\alpha-2$.} See Figure \ref{F2}.

    \vskip0.25cm

    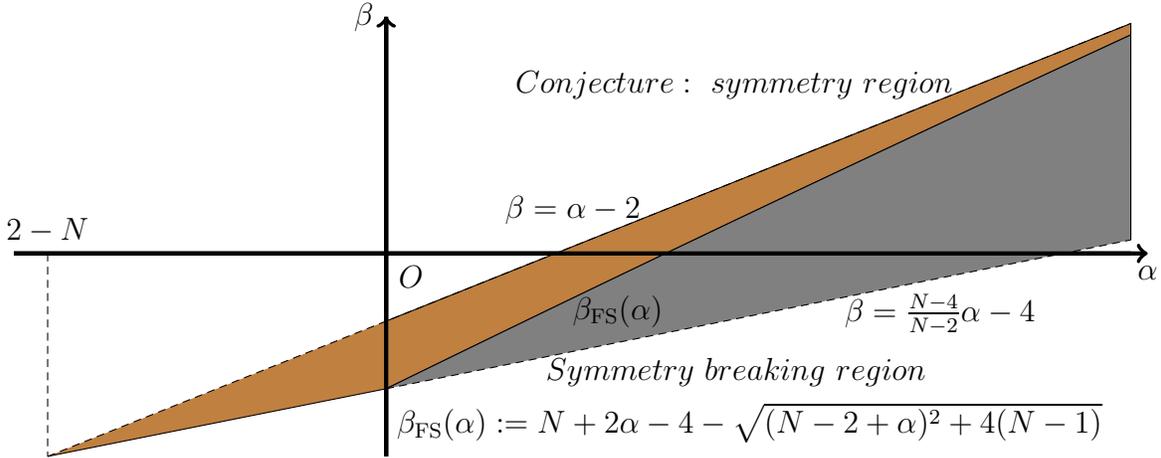
\begin{figure}[ht]
    \begin{tikzpicture}[scale=4.5]
		\draw[->,ultra thick](-1.1,0)--(0,0)node[below right]{$O$}--(2.25,0)node[below]{$\alpha$};
		\draw[->,ultra thick](0,-0.6)--(0,0.7)node[left]{$\beta$};
		
\draw[fill=brown,domain=0:2.2]plot(\x,{0.49*\x+0.2-0.6*((0.15*\x)^2+1)^0.5})
--(2.2,0.68)--(0,-0.2);
\draw[fill=gray,domain=0:2.2]plot(\x,{0.49*\x+0.2-0.6*((0.15*\x)^2+1)^0.5})
--(2.2,0.04);
\draw[fill=brown,domain=-1:0](-1,-0.6)--(0,-0.4)--(0,-0.2);


        \draw[densely dashed](-1,0)node[above]{$2-N$}--(-1,-0.6);
        \draw[densely dashed](0,-0.4)--(2.2,0.04);
        \draw[densely dashed](-1,-0.6)--(2.2,0.68);
        \draw[-,ultra thick](-1.1,0)--(2.2,0);
        \draw[-,ultra thick](0,-0.6)--(0,0.7);

		\node[left] at(1.95,-0.18){$\beta=\frac{N-4}{N-2}\alpha-4$};
        \node[right] at (0.32,0.13){$\beta=
        \alpha-2$};
        \node[left] at (0.85,-0.17){$\beta_{\mathrm{FS}}(\alpha)$};
        \node[right] at (0,-0.5){$\beta_{\mathrm{FS}}(\alpha):=
        N+2\alpha-4-\sqrt{(N-2+\alpha)^2+4(N-1)}$};
        \node[right] at (0.44,-0.35){$Symmetry\ breaking\ region$};
        \node[right] at (0.35,0.5){$Conjecture:\ symmetry\ region$};
\end{tikzpicture}
\caption{\small The new second-order case. We also call the {\em Felli-Schneider region}, or symmetry breaking region, appears in dark grey and is defined by $\alpha>0$ and $\frac{N-4}{N-2}\alpha-4<\beta<\beta_{\mathrm{FS}}(\alpha)$. And we conjecture that the symmetry holds in the brown region defined by $2-N<\alpha\leq 0$ and $\frac{N-4}{N-2}\alpha-4\leq \beta< \alpha-2$ (except for $\alpha=0$ and $\beta=-4$), also $\alpha>0$ and $ \beta_{\mathrm{FS}}(\alpha)\leq \beta<\alpha-2$.}
\label{F2}
\end{figure}

Although so far we can not provide a complete proof of this conjecture, instead, we will give a partial symmetry result when $\beta=\frac{N-4}{N-2}\alpha-4$ with $\alpha<0$ which means $2^{**}_{\alpha,\beta}=2^{**}$ and $\gamma=2\cdot2^{**}+\frac{N^2}{(N-2)(N-4)}\alpha$, where the existence of minimizers has been proved in Theorem \ref{thmene}(iv).

    \begin{theorem}\label{thm2ps}
    Assume that $N\geq 5$ and $2-N<\alpha<0$. For all $u\in \mathcal{D}^{2,2}_{\alpha,\frac{N-4}{N-2}\alpha-4}(\mathbb{R}^N)$,
    \begin{equation}\label{ckn2ps}
    \int_{\mathbb{R}^N}|x|^{4-\frac{N-4}{N-2}\alpha}|\mathrm{div} (|x|^{\alpha}\nabla u)|^2 \mathrm{d}x
    \geq \left(1+\frac{\alpha}{N-2}\right)^{4-\frac{4}{N}}\mathcal{S}_0
    \left(\int_{\mathbb{R}^N}
    |x|^{2\cdot2^{**}+\frac{N^2}{(N-2)(N-4)}\alpha}|u|^{2^{**}} \mathrm{d}x\right)^{\frac{2}{2^{**}}},
    \end{equation}
    where $\mathcal{S}_0$ is the classical best constant of second-order Sobolev inequality given as in \eqref{cssi}. Furthermore, the constant $\left(1+\frac{\alpha}{N-2}\right)^{4-\frac{4}{N}}\mathcal{S}_0$ is sharp and equality holds if and only if $u(x)=cU_\lambda$
    for all $c\in\mathbb{R}$ and $\lambda>0$, where $U_\lambda(x)=\lambda^{\frac{N}{2}(1+\frac{\alpha}{N-2})}U(\lambda x)$ is as in Theorem \ref{thmpwh} replacing $\beta$ by $\frac{N-4}{N-2}\alpha-4$.
    \end{theorem}

    Once the symmetry result is established, an important task is investigating the stability of extremal functions for this inequality, see \cite{BWW03,BE91,CFW13} for examples.

    \begin{theorem}\label{thmafsr}
    Assume that $N\geq 5$ and $2-N<\alpha<0$. There exists a constant $\mathcal{B}=\mathcal{B}(N,\alpha)>0$ such that for all $u\in \mathcal{D}^{2,2}_{\alpha,\frac{N-4}{N-2}\alpha-4}(\mathbb{R}^N)$,
    \begin{align*}
    & \int_{\mathbb{R}^N}|x|^{4-\frac{N-4}{N-2}\alpha}|\mathrm{div} (|x|^{\alpha}\nabla u)|^2 \mathrm{d}x
    -\left(1+\frac{\alpha}{N-2}\right)^{4-\frac{4}{N}}\mathcal{S}_0
    \left(\int_{\mathbb{R}^N}
    |x|^{2\cdot2^{**}+\frac{N^2}{(N-2)(N-4)}\alpha}|u|^{2^{**}} \mathrm{d}x\right)^{\frac{2}{2^{**}}}
    \\
    & \quad \geq \mathcal{B}\inf_{w\in \mathcal{M}_\alpha}
    \int_{\mathbb{R}^N}|x|^{4-\frac{N-4}{N-2}\alpha}|\mathrm{div} (|x|^{\alpha}\nabla (u-w)|^2 \mathrm{d}x,
    \end{align*}
    where $\mathcal{M}_\alpha=\{cU_\lambda: c\in\mathbb{R},\ \lambda>0\}$ is the set of extremal functions of \eqref{ckn2ps}.
    \end{theorem}

    \begin{remark}\label{rpsm}
    The proof of Theorem \ref{thm2ps} is mainly based on the work \cite[Theorem 1.6]{DMY20} which established a Rellich-Sobolev type inequality of Sobolev critical exponent with sharp constant and minimizers of explicit form. The main element is making the change $u(x)=|x|^{-2-\frac{N}{2(N-2)}\alpha}v(x)$ which transforms inequality \eqref{ckn2ps} into \cite[Theorem 1.6]{DMY20}, and the assumption $\alpha<0$ plays a crucial role. The proof of Theorem \ref{thmafsr} is standard, and we see Section \ref{sectps} for details. Once the symmetry holds in other cases, then the same stability result of extremal functions as Theorem \ref{thmafsr} can also be obtained, while if $\beta=\beta_{\mathrm{FS}}(\alpha)$ the stability result should be different as in \cite{FP23}.
    \end{remark}

    \begin{remark}\label{rpsdnw}
    Note that our result Theorem \ref{thm2ps} partially extends the work of Dolbeault, Esteban and Loss \cite{DEL16} for the Sobolev critical exponent case $2^{**}_{\alpha,\beta}=2^{**}$. As for the subcritical case $2^{**}_{\alpha,\beta}<2^{**}$, we believe that, the so-called {\em carr\'{e} du champ} method as in \cite{DEL16,DEL16-arXiv} which transforms the dimension $N$ into higher dimension $M$ not necessary be an integer, will helps us to prove the symmetry region. We will continue to delve deeper into this issue in the forthcoming work..
    \end{remark}

\subsection{Structure of this paper}\label{subsectsop}

 In Section \ref{sectene}, we first prove the singular Sobolev inequality \eqref{ckn2ns} then we show the existence and non-existence of extremal functions for $\mathcal{S}$ by using compactness arguments. Section \ref{sectsbe} is devoting to considering symmetry breaking of extremal functions for sharp constant $\mathcal{S}$. Specifically, in Subsection \ref{sectpmr} we give the uniqueness of radial solutions for Euler-Lagrange equation \eqref{Pwh} by using Emden-Fowler transformation, and also the radial inequality \eqref{Ppbcm} by using a different method which changes the dimension $N$ into $M:=2(N+2\alpha-\beta-4)/(\alpha-\beta-2)$, then in Subsection \ref{sectlp} we will characterize all solutions to the linearized problem \eqref{Pwhl}, and the symmetry breaking conclusion of Theorem \ref{thmmr} will be proved in Subsection \ref{sectsbp}. Section \ref{sectps} is devoted to giving a partial symmetry result about our conjecture and proving Theorem \ref{thm2ps}, and also the stability of extremal functions in Subsection \ref{sectsbr}. Finally, we collect the proof of second-order (CKN) type inequality \eqref{ckn2n} in Appendix \ref{sectpls}.

\section{{\bfseries Existence and non-existence of extremal functions}}\label{sectene}

\subsection{New Sobolev inequality}\label{subsectsr}
Firstly, we prove the symmetry result when $\alpha=0$ and $\beta=-4$ of second-order (CKN) type inequality \eqref{ckn2n} which can be seen as another form of classical second-order Sobolev inequality. It is crucial for the existence and non-existence results about extremal functions of sharp constant $\mathcal{S}$ as in \eqref{ckns}.

\vskip0.25cm

\noindent{\bf\em Proof of Theorem \ref{thmsr}.}
Let $u(x)=|x|^{-2}v(x)\in \mathcal{D}^{2,2}_{0,-4}(\mathbb{R}^N)$. A a direct calculation shows
\[
\Delta u=(8-2N)|x|^{-4}v-4|x|^{-2}(x\cdot\nabla v)+|x|^{-2}\Delta v,
\]
thus,
\begin{align*}
    \int_{\mathbb{R}^N}|x|^4|\Delta u|^2 \mathrm{d}x
    & = \int_{\mathbb{R}^N}
    \Big\{|\Delta v|^2
    + 2 \Delta v[(8-2N)|x|^{-2}v-4|x|^{-2}(x\cdot\nabla v)]
    \\
    & \quad\quad + (8-2N)^2|x|^{-4}v^2
    -8(8-2N)|x|^{-4}v(x\cdot\nabla v)
    + 16|x|^{-4}(x\cdot\nabla v)^2
    \Big\}\mathrm{d}x
    \\
    & = \int_{\mathbb{R}^N}
    \Big\{|\Delta v|^2
    + 4(N-4)|x|^{-2}|\nabla v|^2-8(x\cdot\nabla v)\mathrm{div}(|x|^{-2}\nabla v)
    \Big\}\mathrm{d}x.
    \end{align*}
    We claim that
    \begin{align}\label{eqi}
    (N-4)\int_{\mathbb{R}^N}
    |x|^{-2}|\nabla v|^2\mathrm{d}x
    =
    2\int_{\mathbb{R}^N}
    (x\cdot\nabla v)\mathrm{div}(|x|^{-2}\nabla v)
    \mathrm{d}x.
    \end{align}
    Let $v(x)=|x|w(x)$, then
    \begin{align*}
    \int_{\mathbb{R}^N}
    |x|^{-2}|\nabla v|^2\mathrm{d}x
    =
    \int_{\mathbb{R}^N}
    [|\nabla w|^2
    +|x|^{-2}w^2+
    2|x|^{-2}w(x\cdot\nabla w)]
    \mathrm{d}x,
    \end{align*}
    and
    \begin{align*}
    \int_{\mathbb{R}^N}
    (x\cdot\nabla v)\mathrm{div}(|x|^{-2}\nabla v)
    \mathrm{d}x
    =
    \int_{\mathbb{R}^N}
    \left[
    (N-3)|x|^{-2}(w^2+w(x\cdot\nabla w))
    + \Delta w(w+ x\cdot\nabla w)
    \right]
    \mathrm{d}x.
    \end{align*}
    Therefore, it is easy to verify that \eqref{eqi} is equivalent to
    \begin{align}\label{eqib}
    (N-2)\int_{\mathbb{R}^N}
    |\nabla w|^2\mathrm{d}x
    =
    2\int_{\mathbb{R}^N}
    \Delta w(x\cdot\nabla w)
    \mathrm{d}x.
    \end{align}
    Note that
    \begin{align*}
    \int_{\mathbb{R}^N}
    \Delta w(x\cdot\nabla w)
    \mathrm{d}x
    & = -\int_{\mathbb{R}^N}
    \nabla w\cdot \nabla(x\cdot\nabla w)
    \mathrm{d}x
    \\
    & = -\int_{\mathbb{R}^N}
    \sum^{N}_{i=1}\frac{\partial w}{\partial x_i}
    \left(\frac{\partial w}{\partial x_i}
    +\sum^{N}_{j=1}x_j\frac{\partial^2 w}{\partial x_j\partial x_i}\right)
    \mathrm{d}x
    \\
    & = -\int_{\mathbb{R}^N}|\nabla w|^2\mathrm{d}x
    -\int_{\mathbb{R}^N}
    \sum^{N}_{i=1}\sum^{N}_{j=1}\frac{\partial w}{\partial x_i}x_j\frac{\partial^2 w}{\partial x_j\partial x_i}
    \mathrm{d}x.
    \end{align*}
    By using the partial integration method, we have
    \begin{align*}
    \int_{\mathbb{R}^N}
    \sum^{N}_{i=1}\sum^{N}_{j=1}\frac{\partial w}{\partial x_i}x_j\frac{\partial^2 w}{\partial x_j\partial x_i}
    \mathrm{d}x
    & = \int_{\mathbb{R}^{N-1}}\int_{\mathbb{R}}
    \sum^{N}_{i=1}\sum^{N}_{j=1}\frac{\partial w}{\partial x_i}x_j
    \mathrm{d}\frac{\partial w}{\partial x_i}\mathrm{d}x'
    = -\frac{N}{2}\int_{\mathbb{R}^N}|\nabla w|^2\mathrm{d}x,
    \end{align*}
    where $x=(x',x_j)\in \mathbb{R}^{N-1}\times\mathbb{R}$, thus \eqref{eqib} holds which indicates
    \begin{align}\label{iid}
    \int_{\mathbb{R}^N}|x|^4|\Delta u|^2 \mathrm{d}x
    = \int_{\mathbb{R}^N}|\Delta v|^2\mathrm{d}x.
    \end{align}
    Then by the classical second-order Sobolev inequality, we have
    \begin{align*}
    \int_{\mathbb{R}^N}|\Delta v|^2 \mathrm{d}x
    \geq & \mathcal{S}_0\left(\int_{\mathbb{R}^N}
    |v|^{2^{**}} \mathrm{d}x\right)^{\frac{2}{2^{**}}},
    \end{align*}
    where $\mathcal{S}_0$ is given as in \eqref{cssi}, and the equality holds if and only if
    \begin{equation*}
    v(x)
    =\frac{C\lambda^{\frac{N-4}{2}}}
    {(1+\lambda^{2}|x-x_0|^{2})
    ^{\frac{N-4}{2}}},
    \end{equation*}
    for all $C\in\mathbb{R}$, $\lambda>0$ and $x_0\in\mathbb{R}^N$. Therefore, from \eqref{iid}, we deduce
    \begin{align*}
    \int_{\mathbb{R}^N}|x|^4|\Delta u|^2 \mathrm{d}x
    \geq \mathcal{S}_0\left(\int_{\mathbb{R}^N}
    |x|^{2\cdot2^{**}}|u|^{2^{**}} \mathrm{d}x\right)^{\frac{2}{2^{**}}},
    \end{align*}
    since $\int_{\mathbb{R}^N}
    |x|^{2\cdot2^{**}}|u|^{2^{**}} \mathrm{d}x=\int_{\mathbb{R}^N}|v|^{2^{**}} \mathrm{d}x$, and the equality holds if and only if
    \begin{equation*}
    u(x)
    =\frac{C\lambda^{\frac{N-4}{2}}}
    {|x|^{2}(1+\lambda^{2}|x-x_0|^{2})
    ^{\frac{N-4}{2}}}.
    \end{equation*}
    Now, the proof of Theorem \ref{thmsr} is completed.
    \qed

\vskip0.25cm

Then we are going to prove the non-existence of extremal functions for $\mathcal{S}$ by following the arguments as in \cite{CW01}.

\subsection{Proof of Theorem \ref{thmene}(i): nonexistence for weighted Rellich inequality}\label{subsectpne1}

    When $\beta=\alpha-2$ with $\alpha>2-N$, we know $2^{**}_{\alpha,\beta}=2$, then inequality \eqref{ckn2n} reduces to weighted Rellich inequality
\begin{equation}\label{ckn2nr}
    \int_{\mathbb{R}^N}|x|^{2-\alpha}|\mathrm{div} (|x|^{\alpha}\nabla u)|^2 \mathrm{d}x
    \geq \mathcal{S}\int_{\mathbb{R}^N}
    |x|^{\alpha-2}|u|^{2} \mathrm{d}x, \quad \mbox{for all}\quad u\in C^\infty_0(\mathbb{R}^N\setminus\{0\}),
    \end{equation}
    since $\int_{\mathbb{R}^N}|x|^{2-\alpha}|\mathrm{div} (|x|^{\alpha}\nabla u)|^2 \mathrm{d}x$ is equivalent to $\int_{\mathbb{R}^N}|x|^{2+\alpha}|\Delta u|^2 \mathrm{d}x$ shown as in Proposition \ref{propneq}.
    Next, we will give the explicit form of $\mathcal{S}$.

    At first, we consider the case $\alpha=-2$. Note that
    \begin{align*}
    \int_{\mathbb{R}^N}|x|^{4}|\mathrm{div} (|x|^{-2}\nabla u)|^2 \mathrm{d}x
    & = \int_{\mathbb{R}^N}\left[
    |\Delta u|^2-4|x|^{-2}(x\cdot\nabla u)\Delta u+4|x|^{-4}(x\cdot\nabla u)^2\right]\mathrm{d}x
    \\
    & = \int_{\mathbb{R}^N}\left[
    |\Delta u|^2-2(N-4)|x|^{-2}|\nabla u|^2-4|x|^{-4}(x\cdot\nabla u)^2\right]\mathrm{d}x
    \\
    & \geq \int_{\mathbb{R}^N}\left[
    |\Delta u|^2-2(N-2)|x|^{-2}|\nabla u|^2\right]\mathrm{d}x,
    \end{align*}
    due to
    \[
    (N-2)\int_{\mathbb{R}^N}|\nabla u|^2\mathrm{d}x=2\int_{\mathbb{R}^N}
    \Delta u(x\cdot\nabla u)\mathrm{d}x,
    \]
    which is given as in \eqref{eqib}. Then from the following Rellich type inequalites
    \begin{align*}
    \int_{\mathbb{R}^N}
    |\Delta u|^2\mathrm{d}x\geq \left(\frac{N}{2}\right)^2\int_{\mathbb{R}^N}
    \frac{|\nabla u|^2}{|x|^2}\mathrm{d}x\quad\mbox{and}\quad \int_{\mathbb{R}^N}
    |\Delta u|^2\mathrm{d}x\geq \left(\frac{N(N-4)}{4}\right)^2\int_{\mathbb{R}^N}
    \frac{|u|^2}{|x|^4}\mathrm{d}x,
    \end{align*}
    which are not achieved (see \cite{TZ07}), we obtain
    \begin{align*}
    \int_{\mathbb{R}^N}|x|^{4}|\mathrm{div} (|x|^{-2}\nabla u)|^2 \mathrm{d}x
    \geq & \left(\frac{N-4}{N}\right)^2\int_{\mathbb{R}^N}
    |\Delta u|^2\mathrm{d}x
    \geq \left(\frac{N-4}{2}\right)^4\int_{\mathbb{R}^N}
    \frac{|u|^2}{|x|^4}\mathrm{d}x,
    \end{align*}
    thus $\mathcal{S}\geq \left(\frac{N-4}{2}\right)^4$. The same argument as that in \cite{Ca20}, we take a test function
    \[
    u_\epsilon(x)=|x|^{-\frac{N-4}{2}+\epsilon}g(|x|),
    \]
    where $g\in C^\infty_0(\mathbb{R})$ is a cut-off function such $g(r)=1$ if $|r|\leq 1$ and $g(r)=0$ if $|r|\geq 2$, then we have
    \begin{align*}
    \frac{\int_{\mathbb{R}^N}|x|^{4}|\mathrm{div} (|x|^{-2}\nabla u_\epsilon)|^2 \mathrm{d}x}{\int_{\mathbb{R}^N}
    |x|^{-4}|u_\epsilon|^{2} \mathrm{d}x}\to \left(\frac{N-4}{2}\right)^4_{+},\quad \mbox{as}\quad \epsilon\to 0^+,
    \end{align*}
    therefore we infer $\mathcal{S}= \left(\frac{N-4}{2}\right)^4$ and it is not achieved. That is, we get a new Rellich inequality: when $N\geq 5$, for all $u\in C^\infty_0(\mathbb{R}^N\setminus\{0\})$, it holds that
    \begin{align}\label{nri}
    \int_{\mathbb{R}^N}|x|^{4}|\mathrm{div} (|x|^{-2}\nabla u)|^2 \mathrm{d}x
    \geq \left(\frac{N-4}{2}\right)^4\int_{\mathbb{R}^N}
    \frac{|u|^2}{|x|^4}\mathrm{d}x,
    \end{align}
    and the constant $\left(\frac{N-4}{2}\right)^4$ is sharp  and not achieved for nonzero $u$.
    For general $\alpha\neq -2$, by taking the change $u(x)=|x|^{-\frac{2+\alpha}{2}}v(x)$, with tedious calculations we can also obtain the explicit form of $\mathcal{S}$ by using the new test function $
    v_\epsilon(x)=|x|^{-\frac{N+\alpha-2}{2}+\epsilon}g(|x|)$ as $\epsilon\to 0^+$, of cause $\mathcal{S}$ is also not achieved. In fact,
    \begin{align*}
    \int_{\mathbb{R}^N}|x|^{2-\alpha}|\mathrm{div} (|x|^{\alpha}\nabla u)|^2 \mathrm{d}x
    & = \int_{\mathbb{R}^N}
    \Big[ |\Delta v|^2+2(Q-N+4)|x|^{-2}|\nabla v|^2
    \\
    & \quad\quad -4|x|^{-4}(x\cdot \nabla v)^2+Q^2|x|^{-4}v^2\Big]
    \mathrm{d}x,
    \end{align*}
    where $Q=\frac{2+\alpha}{2}
    \left(N-2+\alpha-\frac{2+\alpha}{2}\right)$. Therefore when $2(Q-N+2)\geq 0$ that is $\alpha\geq 2-N+\sqrt{(N-2)^2+4}$, we can choose \begin{align*}
    \mathcal{S}_1(\alpha)=\left(\frac{N(N-4)}{4}\right)^2
    +2(Q-N+2)\left(\frac{N-4}{2}\right)^2
    +Q^2,
    \end{align*}
    such that inequality \eqref{ckn2nr} holds. Furthermore, \cite[(1.11)]{DMY20} indicates
    \begin{align*}
    \int_{\mathbb{R}^N}
    |\Delta v|^2\mathrm{d}x
    -\frac{N^2-4N+8}{2}\int_{\mathbb{R}^N}|x|^{-2}|\nabla v|^2
    \mathrm{d}x
    +\left(\frac{N-4}{2}\right)^4
    \int_{\mathbb{R}^N}|x|^{-4}v^2
    \mathrm{d}x
    \geq 0,
    \end{align*}
    and the inequality is sharp and strict for any nonzero $v$. Thus when $2(Q-N+2)<0$, that is $2-N<\alpha< 2-N+\sqrt{(N-2)^2+4}$, we can choose
    \begin{align*}
    \mathcal{S}_2(\alpha)=\left(\frac{N(N-4)}{4}\right)^2
    \left(1+\frac{4(Q-N+2)}{N^2-4N+8}\right)
    +\left(\frac{N-4}{2}\right)^4
    \frac{4(Q-N+2)}{N^2-4N+8}
    +Q^2,
    \end{align*}
    due to $1+\frac{4(Q-N+2)}{N^2-4N+8}\geq 0$ is equivalent to $(\alpha+N-2)^2\geq 0$, such that inequality \eqref{ckn2nr} holds. Note that $\mathcal{S}_1(\alpha)=\mathcal{S}_2(\alpha)$ for all $\alpha>2-N$.

    Let
    \[
    h(\alpha):=\left(\frac{N(N-4)}{4}\right)^2
    +2(Q-N+2)\left(\frac{N-4}{2}\right)^2
    +Q^2,
    \]
    then a simple calculation indicates
    \[
    \inf_{\alpha>2-N}h(\alpha)
    =h(2-N)=\frac{(N-2)(N-4)^2}{4}>0.
    \]
    To sum up, for any $\alpha> 2-N$ we can always choose \begin{align}\label{nrinl3}
    \mathcal{S}
    =\left(\frac{N(N-4)}{4}\right)^2
    +2(Q-N+2)\left(\frac{N-4}{2}\right)^2
    +Q^2>0,
    \end{align}
    such that inequality \eqref{ckn2nr} holds, moreover, the inequality is optimal and strict for any nonzero functions.
    \qed

\vskip0.25cm

Then we consider the critical case: $\beta=\frac{N-4}{N-2}\alpha -4$ with $\alpha>0$ which indicates $2^{**}_{\alpha,\beta}= 2^{**}$ and $\gamma=2\cdot2^{**}+\frac{N^2}{(N-2)(N-4)}\alpha$.

\begin{lemma}\label{lemcnee}
Let $N\geq 5$, $\beta=\frac{N-4}{N-2}\alpha -4$ with $\alpha>0$. Then $\mathcal{S}=\mathcal{S}_0$.
\end{lemma}

\begin{proof}
Define the functional
\begin{align}\label{deffc}
\mathcal{F}_\alpha(u):=\frac{\int_{\mathbb{R}^N}
|x|^{4-\frac{N-4}{N-2}\alpha}|\mathrm{div} (|x|^{\alpha}\nabla u)|^2 \mathrm{d}x}
    {\left(\int_{\mathbb{R}^N}
    |x|^{2\cdot2^{**}+\frac{N^2}{(N-2)(N-4)}\alpha}|u|^{2^{**}} \mathrm{d}x\right)^{\frac{2}{2^{**}}}},\quad u\in\mathcal{D}^{2,2}_{\alpha,\frac{N-4}{N-2}\alpha -4}(\mathbb{R}^N).
\end{align}
At first, we will show $\mathcal{F}_\alpha(u)> \mathcal{F}_0(v)$ for all $u\in\mathcal{D}^{2,2}_{\alpha,\frac{N-4}{N-2}\alpha -4}(\mathbb{R}^N)\setminus\{0\}$, where $v(x)=|x|^{\frac{N\alpha}{2(N-2)}}u(x)$. Indeed, by using direct tedious calculations, we deduce
\begin{align*}
\int_{\mathbb{R}^N}
    |x|^{2\cdot2^{**}+\frac{N^2}{(N-2)(N-4)}\alpha}|u|^{2^{**}} \mathrm{d}x
    =\int_{\mathbb{R}^N}
    |x|^{2\cdot2^{**}}|v|^{2^{**}} \mathrm{d}x,
\end{align*}
and
\begin{align*}
\int_{\mathbb{R}^N}
|x|^{4-\frac{N-4}{N-2}\alpha}|\mathrm{div} (|x|^{\alpha}\nabla u)|^2 \mathrm{d}x
& = \int_{\mathbb{R}^N}
|x|^{4}\left(\Delta v + A|x|^{-2}v+B|x|^{-2}(x\cdot\nabla v)\right)^2 \mathrm{d}x
\\
& = \int_{\mathbb{R}^N}
|x|^{4}|\Delta v|^2 \mathrm{d}x
+ \int_{\mathbb{R}^N}
\Big[2A|x|^{2}v\Delta v+2B|x|^2(x\cdot\nabla v)\Delta v
\\
& \quad\quad +2AB(x\cdot\nabla v)v
+A^2v^2+B^2(x\cdot\nabla v)^2\Big] \mathrm{d}x,
\end{align*}
where $A:=-\frac{N\alpha}{2(N-2)}
\left[\frac{(N-4)\alpha}{2(N-2)}+(N-2)\right]<0$ and $B:=-\frac{2\alpha}{N-2}<0$. Using the facts
\[
    \int_{\mathbb{R}^N}(x\cdot\nabla v)\Delta v
    \mathrm{d}x
    =\frac{N}{2}\int_{\mathbb{R}^N}
    |x|^2|\nabla v|^2\mathrm{d}x
    -2\int_{\mathbb{R}^N}(x\cdot\nabla v)^2
    \mathrm{d}x,
\]
and
\[
    N\int_{\mathbb{R}^N}
    |x|^2|\nabla v|^2\mathrm{d}x
    =
    2\int_{\mathbb{R}^N}(x\cdot\nabla v)\mathrm{div}(|x|^2\nabla v)
    \mathrm{d}x,
\]
we infer that
\begin{align*}
& \int_{\mathbb{R}^N}
\Big[2A|x|^{2}v\Delta v+2B|x|^2(x\cdot\nabla v)\Delta v
+2AB(x\cdot\nabla v)v
+A^2v^2+B^2(x\cdot\nabla v)^2\Big] \mathrm{d}x
\\
& = \int_{\mathbb{R}^N}
\Big[A[A-(B-2)N]v^2+(BN-2A)|x|^2|\nabla v|^2+(B^2-4B)(x\cdot\nabla v)^2\Big] \mathrm{d}x
\\
& \geq\left(\frac{(BN-2A+B^2-4B)N^2}{4}+A[A-(B-2)N]\right)
\int_{\mathbb{R}^N}v^2\mathrm{d}x,
\end{align*}
due to $BN-2A>0$, $B^2-4B>0$ and
\[
\int_{\mathbb{R}^N}|x|^2|\nabla v|^2\mathrm{d}x
\geq \int_{\mathbb{R}^N}(x\cdot\nabla v)^2\mathrm{d}x
\geq \frac{N^2}{4}\int_{\mathbb{R}^N}v^2\mathrm{d}x,
\]
the second inequality is easily obtained by using standard spherical decomposition and classical Hardy inequality. It is not difficult to verify that
\[
\xi:=\frac{(BN-2A+B^2-4B)N^2}{4}+A[A-(B-2)N]>0
\]
which is equivalent to
\[
\alpha\left[\alpha^3+4(N-2)\alpha^2+6(N-2)^2\alpha+4(N-2)^3\right]>0,
\]
this is true for all $\alpha>0$. Therefore, $\mathcal{S}\geq \mathcal{S}_0$ for all $\alpha>0$.

Then let us take the test function $u_{\lambda,x_0}\in \mathcal{D}^{2,2}_{0,-4}(\mathbb{R}^N)$ which is the extremal function of inequality \eqref{ckn2ns} given as in \eqref{pbcm0-4}, and note that $|x|^{-\frac{N\alpha}{2(N-2)}}u_{\lambda,x_0}\in \mathcal{D}^{2,2}_{\alpha,\frac{N-4}{N-2}\alpha -4}(\mathbb{R}^N)$ for all $\alpha>2-N$. For $x_0\neq 0$ by a direct computation we get
    \[
    \lim_{\lambda\to +\infty}\mathcal{F}_\alpha
    (|x|^{-\frac{N\alpha}{2(N-2)}}u_{\lambda,x_0})
    =\mathcal{S}_0.
    \]
    Due to this fact one conclude that for $\alpha>2-N$, $\mathcal{S}\leq \mathcal{S}_0$.
    Thus for $\alpha>0$ we infer $\mathcal{S}= \mathcal{S}_0$.
\end{proof}

\subsection{Proof of Theorem \ref{thmene}(ii): nonexistence for critical case}\label{subsectpnec}

We consider the case $\beta=\frac{N-4}{N-2}\alpha-4$ with $\alpha>0$. If the conclusion is not true, for some $\alpha>0$ and $u\in \mathcal{D}^{2,2}_{\alpha,\frac{N-4}{N-2}\alpha -4}(\mathbb{R}^N)\setminus\{0\}$ we get $\mathcal{F}_\alpha(u)=\mathcal{S}$. But using $\mathcal{F}_\alpha(u)>\mathcal{F}_0(v)\geq \mathcal{S}_0$ with $v(x)=|x|^{\frac{N\alpha}{2(N-2)}}u(x)\in \mathcal{D}^{2,2}_{0,-4}(\mathbb{R}^N)$, we get a contradiction to Lemma \ref{lemcnee}.
    \qed

\vskip0.25cm

Then we are going to prove the existence of extremal functions for $\mathcal{S}$ by following the arguments as those in \cite{CM11}. We need the following result.

\begin{lemma}\label{lemer}
 Assume that \eqref{cknc}-\eqref{cknc1} hold and $\frac{N-4}{N-2}\alpha-4 < \beta<\alpha -2$, or $\beta=\frac{N-4}{N-2}\alpha-4$ with $\mathcal{S}<\mathcal{S}_0$. Let $\{u_n\}\subset \mathcal{D}^{2,2}_{\alpha,\beta}(\mathbb{R}^N)\setminus\{0\}$ be a minimizing sequence for $\mathcal{S}$. If $u_n\rightharpoonup u$ weakly in $\mathcal{D}^{2,2}_{\alpha,\beta}(\mathbb{R}^N)$ for some $u\neq 0$, then $u$ is a minimizer for $\mathcal{S}$ and $u_n\to u$ strongly in $\mathcal{D}^{2,2}_{\alpha,\beta}(\mathbb{R}^N)$.
\end{lemma}

\begin{proof}
The proof is standard, one can adapt to our situation a well know argument, see \cite[Chapt. 1, Sect. 4]{St08}. In fact, when $\frac{N-4}{N-2}\alpha-4 < \beta<\alpha -2$ then $2<2^{**}_{\alpha,\beta}<2^{**}$ the subcritical case, by Lion's result \cite{Li85-2} with minor changes the conclusion follows. When $\beta=\frac{N-4}{N-2}\alpha-4$ then $2^{**}_{\alpha,\beta}=2^{**}$ the critical case, by using the same arguments as in \cite[Theorem 4]{WW00}, we can obtain a minimizer under the the assumption $\mathcal{S}<\mathcal{S}_0$.
\end{proof}

To prove the existence results stated in the introduction, we need an $\varepsilon$-compactness criterion for sequences of approximating solutions to \eqref{Pwh}. We start by pointing out an immediate consequence of Rellich Theorem.

\begin{lemma}\label{lemrt}
Let $\Omega$ be a domain with compact closure in $\mathbb{R}^N\setminus\{0\}$, then $\mathcal{D}^{2,2}_{\alpha,\beta}(\Omega)$ is compactly embedded into $H^1(\Omega)$, where $H^1(\Omega)$ denotes the classical Hilbert space.
\end{lemma}

In the next result, let $(\mathcal{D}^{2,2}_{\alpha,\beta}(\mathbb{R}^N))^{-1}$ be the topological dual space of $\mathcal{D}^{2,2}_{\alpha,\beta}(\mathbb{R}^N)$. Since $\mathcal{S}$ is strictly large than zero, we can fix a small $\varepsilon_0>0$ such that
\begin{align}\label{ssb}
\mathcal{S}>
\varepsilon_0^{\frac{2^{**}_{\alpha,\beta}-2}{2^{**}_{\alpha,\beta}}}.
\end{align}

\begin{proposition}\label{propsc0}
Let $\{u_n\}\subset \mathcal{D}^{2,2}_{\alpha,\beta}(\mathbb{R}^N)\setminus\{0\}$, $\{f_n\}\subset (\mathcal{D}^{2,2}_{\alpha,\beta}(\mathbb{R}^N))^{-1}$ be given sequences,
such that $f_n\to 0$ strongly in $\mathcal{D}^{2,2}_{\alpha,\beta}(\mathbb{R}^N)$, $u_n\rightharpoonup 0$ weakly in $\mathcal{D}^{2,2}_{\alpha,\beta}(\mathbb{R}^N)$ and moreover
\begin{align}\label{sc01}
    \mathrm{div}(|x|^{\alpha}\nabla(|x|^{-\beta}
    \mathrm{div}(|x|^\alpha\nabla u_n)))=|x|^\gamma|u_n|^{2^{**}_{\alpha,\beta}-2}u_n+f_n,
\end{align}
and
\begin{align}\label{sc02}
    \int_{B_R}|x|^{\gamma}|u_n|^{2^{**}_{\alpha,\beta}} \mathrm{d}x
    \leq \varepsilon_0,
\end{align}
for some $R>0$, where $\varepsilon_0>0$ satisfies \eqref{ssb}. Here $B_R$ denotes the ball with radius $R$ at origin. Then
\begin{align*}
    \int_{B_r}|x|^{\gamma}|u_n|^{2^{**}_{\alpha,\beta}} \mathrm{d}x \to 0,\quad \mbox{for any}\quad r\in (0,R).
\end{align*}
\end{proposition}

\begin{proof}
Fix $r\in (0,R)$ and take a cut-off function $\varphi\in C^\infty_0(B_R)$ such that $\varphi\equiv 1$ in $B_r$. Notice that
\[
\int_{\mathbb{R}^N}|x|^{-\beta}\mathrm{div} (|x|^{\alpha}\nabla u_n) \mathrm{div} (|x|^{\alpha}\nabla (u_n\varphi^2)) \mathrm{d}x
= \int_{\mathbb{R}^N}|x|^{-\beta}|\mathrm{div} (|x|^{\alpha}\nabla (u_n\varphi))|^2 \mathrm{d}x
+o_n(1),
\]
by Lemma \ref{lemrt}, as $\varphi$ and its derivatives have compact supports in $\mathbb{R}^N\setminus\{0\}$. Therefore, using $u_n\varphi^2$ as a test function in \eqref{sc01} and using H\"{o}lder's inequality we get
\begin{small}\begin{align*}
    \int_{\mathbb{R}^N}|x|^{-\beta}|\mathrm{div} (|x|^{\alpha}\nabla (u_n\varphi))|^2 \mathrm{d}x
    & = \int_{\mathbb{R}^N}|x|^\gamma|u_n|^{2^{**}_{\alpha,\beta}-2}
    (u_n\varphi)^2\mathrm{d}x+o_n(1)
    \\
    & = \int_{B_R}|x|^\gamma|u_n|^{2^{**}_{\alpha,\beta}-2}
    (u_n\varphi)^2\mathrm{d}x+o_n(1)
    \\
    & \leq \left(\int_{B_R}|x|^\gamma|u_n|^{2^{**}_{\alpha,\beta}}
    \mathrm{d}x\right)^{\frac{2^{**}_{\alpha,\beta}-2}
    {2^{**}_{\alpha,\beta}}}
    \left(\int_{B_R}|x|^\gamma|u_n\varphi|^{2^{**}_{\alpha,\beta}}
    \mathrm{d}x\right)^{\frac{2}{2^{**}_{\alpha,\beta}}}+o_n(1).
\end{align*}\end{small}
The left hand side in the above inequality can be bounded from below by using the definition of $\mathcal{S}$. Thus, from \eqref{sc02} we infer
\begin{align*}
    \left(\int_{B_R}|x|^\gamma|u_n\varphi|^{2^{**}_{\alpha,\beta}}
    \mathrm{d}x\right)^{\frac{2}{2^{**}_{\alpha,\beta}}}
    \mathcal{S}
    \leq \varepsilon_0^{{\frac{2^{**}_{\alpha,\beta}-2}{2^{**}_{\alpha,\beta}}}}
    \left(\int_{B_R}|x|^\gamma|u_n\varphi|^{2^{**}_{\alpha,\beta}}
    \mathrm{d}x\right)^{\frac{2}{2^{**}_{\alpha,\beta}}}
    +o_n(1).
\end{align*}
The conclusion readily follow from \eqref{ssb}, since $\varphi\equiv 1$ in $B_r$.
\end{proof}

\subsection{Proof of Theorem \ref{thmene}(iii): existence for subcritical case}\label{subsectpesc}
Here we consider the subcritical case, that is, $\frac{N-4}{N-2}\alpha-4 < \beta<\alpha -2$ which implies $2<2^{**}_{\alpha,\beta}<2^{**}$.
Using Ekeland's variational principle (see \cite[Chapt. 1, Sect. 5]{St08}) we can find a minimizing sequence $\{u_n\}\subset \mathcal{D}^{2,2}_{\alpha,\beta}(\mathbb{R}^N)\setminus\{0\}$, such that \eqref{sc01} holds for a sequence $f_n\to 0$ in $(\mathcal{D}^{2,2}_{\alpha,\beta}(\mathbb{R}^N))^{-1}$ and such that
\[
\int_{\mathbb{R}^N}|x|^{-\beta}|\mathrm{div} (|x|^{\alpha}\nabla u_n)|^2 \mathrm{d}x
    =\int_{\mathbb{R}^N}|x|^{\gamma}|u_n|^{2^{**}_{\alpha,\beta}} \mathrm{d}x+o_n(1)
    = \mathcal{S}^{\frac{2^{**}_{\alpha,\beta}}
    {2^{**}_{\alpha,\beta}-2}}+o_n(1).
\]
Since $\{u_n\}$ is bounded in $\mathcal{D}^{2,2}_{\alpha,\beta}(\mathbb{R}^N)$, we can assume that $u_n\rightharpoonup u$ weakly in $\mathcal{D}^{2,2}_{\alpha,\beta}(\mathbb{R}^N)$. Up to a rescaling, we can also assume that
\begin{align}\label{piia}
\int_{B_2}|x|^{\gamma}|u_n|^{2^{**}_{\alpha,\beta}} \mathrm{d}x
    = \frac{1}{2}\mathcal{S}^{\frac{2^{**}_{\alpha,\beta}}
    {2^{**}_{\alpha,\beta-2}}}.
\end{align}
We claim that $u\neq 0$. Indeed, if $u_n\rightharpoonup 0$ weakly in $\mathcal{D}^{2,2}_{\alpha,\beta}(\mathbb{R}^N)$, then
\[
\int_{B_1}|x|^{\gamma}|u_n|^{2^{**}_{\alpha,\beta}} \mathrm{d}x
    = o_n(1),
\]
by Proposition \ref{propsc0}. On the other hand,
\[
\int_{B_2\setminus B_1}|x|^{\gamma}|u_n|^{2^{**}_{\alpha,\beta}} \mathrm{d}x
    = o_n(1),
\]
by Lemma \ref{lemrt} and Rellich Theorem, contradicting \eqref{piia}. Thus the minimizing sequence $\{u_n\}$ converges weakly to a non-trivial limit, then we can apply Lemma \ref{lemer} to conclude.
\qed

\vskip0.25cm

Next, we focus on the critical case $2^{**}_{\alpha,\beta}=2^{**}$ with $2-N<\alpha<0$.

\begin{lemma}\label{lemesla}
Let $N\geq 5$, $\beta=\frac{N-4}{N-2}\alpha -4$, that is, $2^{**}_{\alpha,\beta}=2^{**}$. If $\mathcal{S}<\mathcal{S}_0$, then $\mathcal{S}$ is achieved.
\end{lemma}

\begin{proof}
We select a minimizing sequence $\{u_n\}$ as previous proof. In particular, there exists a sequence $f_n\to 0$ in $(\mathcal{D}^{2,2}_{\alpha,\beta}(\mathbb{R}^N))^{-1}$ such that \begin{small}\begin{align}\label{esla1}
    \mathrm{div}(|x|^{\alpha}\nabla(|x|^{4-\frac{N-4}{N-2}\alpha}
    \mathrm{div}(|x|^\alpha\nabla u_n)))
    & = |x|^{2\cdot2^{**}+\frac{N^2}{(N-2)(N-4)}\alpha}|u_n|^{2^{**}-2}u_n
    +f_n,\\
\label{esla2}
\int_{\mathbb{R}^N}|x|^{4-\frac{N-4}{N-2}\alpha}|\mathrm{div} (|x|^{\alpha}\nabla u_n)|^2 \mathrm{d}x
    & = \int_{\mathbb{R}^N}
    |x|^{2\cdot2^{**}+\frac{N^2}{(N-2)(N-4)}\alpha}
    |u_n|^{2^{**}} \mathrm{d}x+o_n(1)
    \nonumber\\
    & = \mathcal{S}^{\frac{2^{**}}{2^{**}-2}}+o_n(1),
        \\
\label{esla3}
\int_{B_2}|x|^{2\cdot2^{**}+\frac{N^2}{(N-2)(N-4)}\alpha}
|u_n|^{2^{**}} \mathrm{d}x
    & = \frac{1}{2}\mathcal{S}^{\frac{2^{**}}{2^{**}-2}}.
\end{align}\end{small}
As before, we have to prove that $\{u_n\}$ cannot converge weakly to $0$. By contradiction, we assume that $u_n\rightharpoonup 0$ weakly in $\mathcal{D}^{2,2}_{\alpha,\frac{N-4}{N-2}\alpha -4}(\mathbb{R}^N)$. Then we can argue as in previous proof to get
\[
\int_{B_1}|x|^{2\cdot2^{**}+\frac{N^2}{(N-2)(N-4)}\alpha}
|u_n|^{2^{**}} \mathrm{d}x
    = o_n(1),
\]
and hence
\begin{align}\label{eslaa}
\int_{B_2\setminus B_1}
|x|^{2\cdot2^{**}+\frac{N^2}{(N-2)(N-4)}\alpha}
|u_n|^{2^{**}}\mathrm{d}x
    = \frac{1}{2}\mathcal{S}^{\frac{2^{**}}{2^{**}-2}}+o_n(1),
\end{align}
Now we take a cut-off function $\varphi\in C^\infty_0(\mathbb{R}^N\setminus\{0\})$ such that $\varphi\equiv 1$ in $B_2\setminus B_1$ and we use $u_n\varphi^2$ as a test function in \eqref{esla1}. Using Lemma \ref{lemrt}, H\"{o}lder's inequality and \eqref{esla2}, we have
\begin{align*}
\int_{\mathbb{R}^N}|x|^{4-\frac{N-4}{N-2}\alpha}|\mathrm{div} (|x|^{\alpha}\nabla (u_n\varphi))|^2 \mathrm{d}x
    \leq \mathcal{S} \left(\int_{\mathbb{R}^N}
    |x|^{2\cdot2^{**}+\frac{N^2}{(N-2)(N-4)}\alpha}
    |u_n\varphi|^{2^{**}} \mathrm{d}x\right)^{\frac{2}{2^{**}}}+o_n(1).
\end{align*}
Note that by using Lemma \ref{lemrt} and Rellich Theorem, we know the function
\begin{align*}
\Delta(|x|^{\frac{t}{2}}u_n\varphi)
-|x|^{\frac{t}{2}}\Delta(u_n\varphi)
=  \frac{t}{2}\left(\frac{t}{2}-2+N\right)|x|^{\frac{t}{2}-2}u_n\varphi
+t|x|^{\frac{t}{2}-2}(x\cdot\nabla(u_n\varphi)),\quad \forall t\in\mathbb{R},
\end{align*}
converges to $0$ strongly in $L^2(\mathbb{R}^N)$ (see \cite{CM11}). Then let us define
\begin{align*}
\widetilde{F}_n:
& = |x|^2\Delta(|x|^{\frac{N\alpha}{2(N-2)}}u_n\varphi)
-|x|^{2-\frac{(N-4)\alpha}{2(N-2)}}
\mathrm{div} (|x|^{\alpha}\nabla (u_n\varphi))
\\
& = |x|^2\Delta(|x|^{\frac{N\alpha}{2(N-2)}}u_n\varphi)
-\alpha|x|^{\frac{N\alpha}{2(N-2)}}(x\cdot\nabla(u_n\varphi))
-|x|^{2+\frac{N\alpha}{2(N-2)}}\Delta(u_n\varphi)
\\
& = \frac{N\alpha}{2(N-2)}\left[\frac{N\alpha}{2(N-2)}-2+N\right]
|x|^{\frac{N\alpha}{2(N-2)}}u_n\varphi
+\frac{2\alpha}{N-2}|x|^{\frac{N\alpha}{2(N-2)}}(x\cdot\nabla(u_n\varphi)),
\end{align*}
thus we get that $\widetilde{F}_n\to 0$ strongly in $L^2(\mathbb{R}^N)$. Therefore, by the new Sobolev inequality established in Theorem \ref{thmsr},
\begin{align*}
    \int_{\mathbb{R}^N}|x|^{4-\frac{N-4}{N-2}\alpha}|\mathrm{div} (|x|^{\alpha}\nabla (u_n\varphi))|^2 \mathrm{d}x
    & = \int_{\mathbb{R}^N}
    |x|^4\Delta(|x|^{\frac{N\alpha}{2(N-2)}}u_n\varphi)\mathrm{d}x
    +o_n(1)
    \\
    & \geq \mathcal{S}_0\left(\int_{\mathbb{R}^N}
    |x|^{2\cdot2^{**}}||x|^{\frac{N\alpha}{2(N-2)}}u_n\varphi|^{2^{**}} \mathrm{d}x\right)^{\frac{2}{2^{**}}}+o_n(1)
    \\
    & = \mathcal{S}_0\left(\int_{\mathbb{R}^N}
    |x|^{2\cdot2^{**}+\frac{N^2}{(N-2)(N-4)}\alpha}
    |u_n\varphi|^{2^{**}} \mathrm{d}x\right)^{\frac{2}{2^{**}}}+o_n(1).
\end{align*}
Then we conclude that
\begin{small}\begin{align*}
    \mathcal{S}_0\left(\int_{\mathbb{R}^N}
    |x|^{2\cdot2^{**}+\frac{N^2}{(N-2)(N-4)}\alpha}
    |u_n\varphi|^{2^{**}} \mathrm{d}x\right)^{\frac{2}{2^{**}}}
    \leq \mathcal{S}\left(\int_{\mathbb{R}^N}
    |x|^{2\cdot2^{**}+\frac{N^2}{(N-2)(N-4)}\alpha}
    |u_n\varphi|^{2^{**}} \mathrm{d}x\right)^{\frac{2}{2^{**}}}+o_n(1).
\end{align*}\end{small}
Thus
\begin{align*}
    o_n(1)
    =\int_{\mathbb{R}^N}
    |x|^{2\cdot2^{**}+\frac{N^2}{(N-2)(N-4)}\alpha}
    |u_n\varphi|^{2^{**}} \mathrm{d}x
    \geq \int_{B_2\setminus B_1}
|x|^{2\cdot2^{**}+\frac{N^2}{(N-2)(N-4)}\alpha}
|u_n|^{2^{**}}\mathrm{d}x,
\end{align*}
as $0<\mathcal{S}<\mathcal{S}_0$ by assumption and $\varphi\equiv 1$ in the annulus $B_2\setminus B_1$. Since this conclusion contradicts \eqref{eslaa}, we infer that the weak limit of the minimizing sequence $\{u_n\}$ cannot vanish. Then we can apply Lemma \ref{lemer} to conclude.
\end{proof}

\begin{lemma}\label{lemcess}
Let $N\geq 5$, $\beta=\frac{N-4}{N-2}\alpha -4$ with $2-N<\alpha<0$. Then $\mathcal{S}<\mathcal{S}_0$.
\end{lemma}

\begin{proof}
We use the same general ideas as in the proof of Lemma \ref{lemcnee}.
At first, we will show the functional $\mathcal{F}_\alpha$ defined in \eqref{deffc} satisfying $\mathcal{F}_\alpha(u)< \mathcal{F}_0(v)$ for all $u\in\mathcal{D}^{2,2}_{\alpha,\frac{N-4}{N-2}\alpha -4}(\mathbb{R}^N)\setminus\{0\}$, where $v(x)=|x|^{\frac{N\alpha}{2(N-2)}}u(x)$. The same proof as that of Lemma \ref{lemcnee}, we obtain
\begin{align*}
\int_{\mathbb{R}^N}
|x|^{4-\frac{N-4}{N-2}\alpha}|\mathrm{div} (|x|^{\alpha}\nabla u)|^2 \mathrm{d}x
& = \int_{\mathbb{R}^N}
|x|^{4}|\Delta v|^2 \mathrm{d}x
+ \int_{\mathbb{R}^N}
\Big[2A|x|^{2}v\Delta v+2B|x|^2(x\cdot\nabla v)\Delta v
\\
& \quad\quad +2AB(x\cdot\nabla v)v
+A^2v^2+B^2(x\cdot\nabla v)^2\Big] \mathrm{d}x,
\end{align*}
and
\begin{align*}
& \int_{\mathbb{R}^N}
\Big[2A|x|^{2}v\Delta v+2B|x|^2(x\cdot\nabla v)\Delta v
+2AB(x\cdot\nabla v)v
+A^2v^2+B^2(x\cdot\nabla v)^2\Big] \mathrm{d}x
\\
& = \int_{\mathbb{R}^N}
\Big[A[A-(B-2)N]v^2+(BN-2A)|x|^2|\nabla v|^2+(B^2-4B)(x\cdot\nabla v)^2\Big] \mathrm{d}x
\\
& \leq \left(\frac{(BN-2A+B^2-4B)N^2}{4}+A[A-(B-2)N]\right)
\int_{\mathbb{R}^N}v^2\mathrm{d}x,
\end{align*}
where $A:=-\frac{N\alpha}{2(N-2)}
\left[\frac{(N-4)\alpha}{2(N-2)}+(N-2)\right]>0$ and $B:=-\frac{2\alpha}{N-2}\in (0,2)$,
due to $BN-2A<0$ and $B^2-4B<0$. It is not difficult to verify that
\begin{align}\label{defcl}
\xi:=\frac{(BN-2A+B^2-4B)N^2}{4}+A[A-(B-2)N]<0
\end{align}
is equivalent to
\[
\alpha
(\alpha+N-2)
\left\{(\alpha+N-2)^2+(N-2)[\alpha+2(N-2)]\right\}<0,
\]
this is true for all $2-N<\alpha<0$.

Then since $|x|^{-\frac{N\alpha}{2(N-2)}}u_{1,0}\in \mathcal{D}^{2,2}_{\alpha,\frac{N-4}{N-2}\alpha -4}(\mathbb{R}^N)$ for all $\alpha>2-N$, where $u_{1,0}$ is the extremal function of inequality \eqref{ckn2ns} given as in \eqref{pbcm0-4}, we infer
\begin{small}\begin{align*}
\mathcal{S}\leq \mathcal{F}_\alpha(|x|^{-\frac{N\alpha}{2(N-2)}}u_{1,0})
<\mathcal{F}_0(u_{1,0})
    =\mathcal{S}_0.
\end{align*}\end{small}
The proof is completed.
\end{proof}

\subsection{Proof of Theorem \ref{thmene}(iv): existence for critical case}\label{subsectpec}

We consider the case $\beta=\frac{N-4}{N-2}\alpha-4$ with $2-N<\alpha<0$. The conclusion follows Lemmas \ref{lemesla} and  \ref{lemcess} directly.
\qed

\vskip0.25cm

\section{{\bfseries Symmetry breaking of extremal functions}}\label{sectsbe}

    In this section, we will show a symmetry breaking phenomenon about extremal functions of the second-order (CKN) type inequality \eqref{ckn2n}. In order to obtain this conclusion, firstly in Subsection \ref{sectpmr} we will consider the extremal functions in radial space, then in Subsection \ref{sectlp} we will classify all solutions of the linearized equation related to radial extremal functions, and finally Subsection \ref{sectsbp} we give the proof of Theorem \ref{thmmr} about symmetry breaking phenomenon.

\subsection{{\bfseries Extremal functions in radial case}}\label{sectpmr}

    In order to obtain the uniqueness of radial solution for equation \eqref{Pwh}, we will use the classical Emden-Fowler transformation.
    Following the work of Huang and Wang \cite{HW20}, let $\varphi, \psi\in C^2(\mathbb{R})$ be defined by
    \begin{align}\label{l-et}
    u(|x|)=|x|^{-\kappa_1}\varphi(t),\quad v(|x|)=|x|^{-\kappa_2}\phi(t),\quad \mbox{with}\ t=-\ln |x|,
    \end{align}
    where
    \begin{align*}
    \kappa_1=\frac{N+\gamma}{2^{**}_{\alpha,\beta}}
    =\frac{N+2\alpha-\beta-4}{2},\quad
    \kappa_2=\beta+2-\alpha+\kappa_1=\frac{N+\beta}{2}.
    \end{align*}
    Since $\alpha>2-N$ then $\kappa_1+\kappa_2=N+\alpha-2>0$.

    Denote the constants
    \begin{align*}
    \mathcal{A}:= & \frac{N+\alpha-2}{2}-\kappa_1
    =-\frac{N+\alpha-2}{2}+\kappa_2
    =\frac{\beta+2-\alpha}{2}, \\
    \mathcal{B}:= & \kappa_1(N+\alpha-2-\kappa_1)=\kappa_1\kappa_2
    =\frac{(N+\beta)(N+2\alpha-\beta-4)}{4}.
    \end{align*}

\vskip0.25cm

\noindent{\bf\em Proof of Theorem \ref{thmpwh}.}
A direct computation shows that a positive radial function $u\in \mathcal{D}^{2,2}_{\alpha,\beta}(\mathbb{R}^N)$ solves \eqref{Pwh} if and only if the functions $\varphi, \phi$ satisfy
    \begin{eqnarray*}
    \left\{ \arraycolsep=1.5pt
       \begin{array}{ll}
        -\varphi''+2\mathcal{A}\varphi'+\mathcal{B}\varphi=\phi\quad \mbox{in}\ \mathbb{R},\\[2mm]
        -\phi''-2\mathcal{A}\phi'+\mathcal{B}\phi
        =|\varphi|^{2^{**}_{\alpha,\beta}-2}\varphi\quad \mbox{in}\ \mathbb{R},
        \end{array}
    \right.
    \end{eqnarray*}
    that is, $\varphi(t)$ satisfies the following fourth-order ordinary differential equation
    \begin{equation}\label{Pwht}
    \varphi^{(4)}-K_2\varphi''+K_0\varphi
    =|\varphi|^{2^{**}_{\alpha,\beta}-2}\varphi \quad \mbox{in}\  \mathbb{R},\quad \varphi\in H^2(\mathbb{R}),
    \end{equation}
    with the constants
    \begin{align*}
    K_2= \lambda_1^2+\lambda_2^2
    =\frac{(N+\alpha-2)^2+(\beta+2-\alpha)^2}{2}, \quad
    K_0= \lambda_1^2\lambda_2^2
    =\frac{(N+\beta)^2(N+2\alpha-\beta-4)^2}{16}.
    \end{align*}
    Here $H^2(\mathbb{R})$ denotes the completion of $C^\infty_0(\mathbb{R})$ with the norm
    \[
    \|\varphi\|_{H^2(\mathbb{R})}
    =\left(\int_{\mathbb{R}}(|\varphi''|^2+K_2|\varphi'|^2
    +K_0|\varphi|^2)\mathrm{d}t\right)^{1/2}.
    \]
    Clearly, $K_2^2\geq 4 K_0$. Then since $2^{**}_{\alpha,\beta}>2$, applying
    \cite[Theorem 2.2 ]{BM12} directly, we could get the existence and uniqueness (up to translations, inversion $t\mapsto -t$ and change of sign) of smooth solution to equation \eqref{Pwht}. We can follow Huang and Wang \cite{HW20} to give the explicit form of $\varphi$.

    Indeed, we can set
    \begin{equation}\label{vtd}
    \varphi(t)=\mathcal{C}(\cosh \nu t)^m,
    \end{equation}
    where $\mathcal{C}, \nu, m$ will be determined later. Here $\cosh s=(e^s+e^{-s})/2$ and $\sinh s=(e^s-e^{-s})/2$. Assume that $\nu>0$ since $\cosh(\cdot)$ is even. By complex and tedious calculation, 
    there holds
    \[
    \varphi'(t)=\mathcal{C} m\nu(\cosh \nu t)^{m-1}\sinh \nu t,
    \]
    then
    \begin{align*}
    \varphi''(t)
    = & \mathcal{C} m\nu^2(\cosh \nu t)^{m}
    + \mathcal{C}m(m-1)\nu^2(\cosh \nu t)^{m-2}(\sinh \nu t)^2
    \\
    = & \mathcal{C}m^2\nu^2(\cosh \nu t)^m-\mathcal{C}m(m-1)\nu^2(\cosh \nu t)^{m-2},
    \end{align*}
    due to $(\sinh \nu t)^2=(\cosh \nu t)^2-1$. Similarly,
    \begin{align*}
    \varphi^{(4)}(t)
    = & \mathcal{C}m^4\nu^4(\cosh \nu t)^m-\mathcal{C}m(m-1)[m^2+(m-2)^2]\nu^4(\cosh \nu t)^{m-2}
    \\
    & +\mathcal{C}m(m-1)(m-2)(m-3)\nu^4(\cosh \nu t)^{m-4}.
    \end{align*}
    Therefore,
    \begin{align*}
    \varphi^{(4)}(t)-K_2\varphi''(t)+K_0\varphi(t)
    = & \mathcal{C}[m^4\nu^4-K_2m^2\nu^2+K_0](\cosh \nu t)^{m}
    \\
    & -\mathcal{C}m(m-1)\left[\left(m^2+(m-2)^2\right)\nu^2-K_2\right]
    (\cosh \nu t)^{m-2}
    \\
    & +\mathcal{C}m(m-1)(m-2)(m-3)\nu^4(\cosh \nu t)^{m-4}.
    \end{align*}
    To make sure that such $\varphi$ is the solution, one needs
    \begin{align}\label{z1}
    m^4\nu^4-K_2m^2\nu^2+K_0=0,\\ \label{z2}
    m(m-1)\left[\left(m^2+(m-2)^2\right)\nu^2-K_2\right]
    =0,
    \end{align}
    and
    \begin{align}\label{z3}
    \mathcal{C}m(m-1)(m-2)(m-3)\nu^4(\cosh \nu t)^{m-4}
    =\mathcal{C}^{2^{**}_{\alpha,\beta}-1}(\cosh \nu t)^{m(2^{**}_{\alpha,\beta}-1)}.
    \end{align}
    From \eqref{z2}-\eqref{z3} we have
    \begin{align}\label{mnc}
    m= & -\frac{4}{2^{**}_{\alpha,\beta}-2}=\frac{N+\beta}{\beta+2-\alpha},
    \quad \nu=\left(\frac{K_2}{m^2+(m-2)^2}\right)^{\frac{1}{2}}
    =\frac{\alpha-\beta-2}{2},\nonumber\\
    \mathcal{C}= & \left[m(m-1)(m-2)(m-3)\nu^4\right]
    ^{\frac{1}{2^{**}_{\alpha,\beta}-2}}.
    \end{align}
    It remains to verify that the above chosen $m, \nu, \mathcal{C}$ satisfy \eqref{z1}. Under the condition $(N+\beta)(N+\gamma)=(N+2\alpha-\beta-4)^2$ one can obtain
    \[
    \frac{K_2}{2\sqrt{K_0}}
    =\frac{(N+\alpha-2)^2+(\beta+2-\alpha)^2}
    {(N+\beta)(N+2\alpha-\beta-4)}
    =\frac{(2^{**}_{\alpha,\beta})^2+4}
    {4\cdot2^{**}_{\alpha,\beta}}
    =\frac{m^2-2m+2}{m(m-2)},
    \]
    which means
    \[
    \nu^2=\frac{K_2}{m^2+(m-2)^2}=\frac{K_2-\sqrt{K_2^2-4K_0}}{2m^2},
    \]
    then $m^4\nu^4-K_2m^2\nu^2+K_0=0$, due to $K_2^2\geq 4 K_0$.

    Now, let us return to the original problem. The change of $u(|x|)=|x|^{-\frac{N+2\alpha-\beta-4}{2}}\varphi(t)$ with $t=-\ln |x|$, then $\varphi(t)=\mathcal{C}(\cosh \nu t)^m$ indicates
    \[
    u(|x|)=\frac{\mathcal{C}}{2^m}|x|^{-(\frac{N+2\alpha-\beta-4}{2}+m\nu)}
    (1+|x|^{2\nu})^m,
    \]
    where $m, \nu, \mathcal{C}$ are given as in \eqref{mnc}. Thus we obtain
    \[
    u(x)=\frac{C_{N,\alpha,\beta}}
    {|x|^{\alpha-\beta-2}(1+|x|^{\alpha-\beta-2})
    ^{\frac{N+\beta}{\alpha-\beta-2}}},
    \]
    where $C_{N,\alpha,\beta}=\left[(N+\beta)(N+\alpha-2)
    (N+2\alpha-\beta-4)(N+3\alpha-2\beta-6)\right]
    ^{\frac{N+\beta}{4(\alpha -\beta-2)}}$. Note that the invariance of $\varphi$ up to translations, inversion $t\mapsto -t$ and change of sign indicates invariance of $u$ up to scalings and change of sign, thus \eqref{Pwh} admits a unique (up to scalings and change of sign)
    radial solution of the form $\pm U_{\lambda}(x)=\pm \lambda^{\frac{N+2\alpha-\beta-4}{2}}U(\lambda x)$ for all $\lambda>0$, where
    \begin{align*}
    U(x)=\frac{C_{N,\alpha,\beta}}
    {|x|^{\alpha-\beta-2}(1+|x|^{\alpha-\beta-2})
    ^{\frac{N+\beta}{\alpha-\beta-2}}},
    \end{align*}
    Now, the proof of Theorem \ref{thmpwh} is completed.
    \qed

    \vskip0.25cm

    Theorem \ref{thmPbcb} can be done directly as those in \cite{HW20} by using the Emden-Fowler transformation. Here we will use different method that transforms the dimension $N$ into $M:=\frac{2(N+2\alpha-\beta-4)}{\alpha-\beta-2}$  which was first introduced in \cite{DGG17} (see also \cite{DGT23-jde}), that is, we make the change of variable $v(s)=r^au(r)$ with $|x|=r=s^q$ where $a=N+\alpha-2$ and $q=\frac{2}{2+\beta-\alpha}$, related to Sobolev inequality to investigate the sharp constants and extremal functions of best constant $\mathcal{S}_r$ as in (\ref{Ppbcm}).

\vskip0.25cm

\noindent{\bf\em Proof of Corollary \ref{thmPbcb}.}
Let $u\in \mathcal{\mathcal{D}}^{2,2}_{\alpha,\beta}(\mathbb{R}^N)$ be radial. Making the change $v(s)=r^au(r)$ with $|x|=r=s^q$, where
    \begin{equation}\label{deft}
    a=N+\alpha-2>0,\quad \mbox{and}\quad q=\frac{2}{2+\beta-\alpha}<0,
    \end{equation}
then we have
    \begin{align*}
    \int_{\mathbb{R}^N}|x|^{-\beta}|\mathrm{div} (|x|^{\alpha}\nabla u)|^2 \mathrm{d}x
    & =\omega_{N-1}\int^\infty_0 \left[u''(r)+\frac{N+\alpha-1}{r}u'(r)\right]^2 r^{N+2\alpha-\beta-1}\mathrm{d}r \\
    & =\omega_{N-1} (-q)^{-3}\int^\infty_0
    \bigg[v''(s)+\frac{(N+\alpha-2-2a)q+1}{s}v'(s)
    \\
    &\quad\quad +\frac{q^2a(a+2-N-\alpha)}{s^2}v(s)\bigg]^2
    s^{q(N+2\alpha-\beta-4-2a)+3}\mathrm{d}s,
    \end{align*}
    where $\omega_{N-1}=2\pi^{\frac{N}{2}}/\Gamma(\frac{N}{2})$ is the surface area for unit ball of $\mathbb{R}^N$.
    By the choice of $a$ and $q$,
    \[
    q^2a(a+2-N-\alpha)=0\quad \mbox{and}\quad  (N+\alpha-2-2a)q+1=q(N+2\alpha-\beta-4-2a)+3.
    \]
    Now, we set
    \begin{equation}\label{defm}
    M:=(N+\alpha-2-2a)q+2
    =\frac{2(N+2\alpha-\beta-4)}{\alpha-\beta-2}>4,
    \end{equation}
    which implies
    \begin{equation*}
    \begin{split}
    \int^\infty_0 \left[u''(r)+\frac{N+\alpha-1}{r}u'(r)\right]^2 r^{N+2\alpha-\beta-1}\mathrm{d}r
    = (-q)^{-3}\int^\infty_0\left[v''(s)+\frac{M-1}{s}v'(s)\right]^2
    s^{M-1}\mathrm{d}s.
    \end{split}
    \end{equation*}
    To sum up,
    \begin{equation}\label{defbcst}
    \mathcal{S}_r
    = (-q)^{\frac{4}{M}-4}
    \omega^{1-\frac{2}{2^{**}_{\alpha,\beta}}}_{N-1}\mathcal{B}(M),
    \end{equation}
    where
    \begin{equation}\label{defbcscm}
    \mathcal{B}(M)=\inf_{v\in \mathcal{\mathcal{D}}^{2,2}_{\infty}(M)\backslash\{0\}}
    \frac{\int^\infty_0\left[v''(s)+\frac{M-1}{s}v'(s)\right]^2
    s^{M-1}\mathrm{d}s}
    {\left(\int^\infty_0|v(s)|^{\frac{2M}{M-4}}s^{M-1}
    \mathrm{d}s \right)^{\frac{M-4}{M}}}.
    \end{equation}
    Here $\mathcal{\mathcal{D}}^{2,2}_{\infty}(M)$ denotes the completion of $C^\infty_0((0,\infty))$ with respect to the norm
    \[
    \|v\|^2_{\mathcal{\mathcal{D}}^{2,2}_{\infty}(M)}=\int^\infty_0
    \left[v''(s)+\frac{M-1}{s}v'(s)\right]^2
    s^{M-1}\mathrm{d}s,
    \]
    see \cite{dS23}. Following the arguments as those in our recent work \cite{DT24}, the proof can be done. In fact, the authors in \cite{dS23} proved that $\mathcal{B}(M)>0$ which can be achieved, and the minimizers are solutions (up to some suitable multiplications) of
    \begin{equation}\label{les}
    \Delta_s^2v=|v|^{\frac{8}{M-4}}v\quad \mbox{in}\ (0,\infty),\quad v\in \mathcal{\mathcal{D}}^{2,2}_{\infty}(M),
    \end{equation}
    where $\Delta_s=\frac{\partial^2}{\partial s^2}+\frac{M-1}{s}\frac{\partial}{\partial s}$. As in the proof of Theorem \ref{thmpwh}, the uniqueness of solutions of \eqref{Pwht} indicates that \eqref{les} admits a unique (up to scalings and change of sign) nontrivial solution of the form
    \[
    \bar{v}(s)=\pm A\tau^{\frac{M-4}{2}}(1+\tau^2s^2)^{-\frac{M-4}{2}}, \quad \mbox{for some suitable}\quad A\in\mathbb{R}^+\quad \mbox{and for all}\quad \tau>0,
    \]
    which implies $\bar{v}$ is the unique (up to scalings and multiplications) minimizer for $\mathcal{B}(M)$. Therefore, putting $\bar{v}$ into \eqref{defbcscm} as a test function, we can directly obtain
    \begin{align*}
    &\mathcal{B}(M)=(M-4)(M-2)M(M+2)
    \left(\frac{\Gamma^2\left(\frac{M}{2}\right)}
    {2\Gamma(M)}\right)^{\frac{4}{M}},
    \end{align*}
    where $\Gamma$ denotes the classical Gamma function.
    Then turning back to \eqref{defbcst}, we have
    \begin{equation*}
    \mathcal{S}_r
    = \left(\frac{2}{\alpha-\beta-2}\right)
    ^{\frac{2(\alpha-\beta-2)}{N+2\alpha-\beta-4}-4}
    \left(\frac{2\pi^{\frac{N}{2}}}{\Gamma(\frac{N}{2})}\right)
    ^{\frac{2(\alpha-\beta-2)}{N+2\alpha-\beta-4}}
    \mathcal{B}\left(\frac{2(N+2\alpha-\beta-4)}{\alpha-\beta-2}\right),
    \end{equation*}
    and it is achieved if and only if by
    \begin{equation*}
    V_{\lambda}(x)
    =\frac{C\lambda^{\frac{N-4-\beta+2\alpha}{2}}|\lambda x|^{-(N+\alpha-2)}}
    {(1+\lambda^{2+\beta-\alpha}|x|^{2+\beta-\alpha})
    ^{\frac{N+\beta}{\alpha-\beta-2}}}
    =\frac{C\lambda^{\frac{N+\beta}{2}}}
    {|x|^{\alpha-\beta-2}(1+\lambda^{\alpha-\beta-2}|x|^{\alpha-\beta-2})
    ^{\frac{N+\beta}{\alpha-\beta-2}}},
    \end{equation*}
    for all $C\in\mathbb{R}\backslash\{0\}$ and $\lambda>0$.
    The proof of Corollary \ref{thmPbcb} is now completed.
    \qed

\vskip0.25cm

\subsection{{\bfseries Classification of linearized problem}}\label{sectlp}

By using the standard spherical decomposition, we can characterize all solutions to the linearized problem (\ref{Pwhl}) which has it own interests.

\vskip0.25cm
\noindent{\bf\em Proof of Theorem \ref{thmpwhl}.}
    We follow the arguments as those in \cite{BWW03}, and also \cite{DT23-rs}.
    Firstly, let us decompose the fourth-order equation (\ref{Pwhl}) into a system of two second-order equations, that is, set
    \begin{equation}\label{PpwhlW}
    \begin{split}
    -\mathrm{div}(|x|^{\alpha}\nabla u)=|x|^\beta w,
    \end{split}
    \end{equation}
    then problem (\ref{Pwhl}) is equivalent to the following system:
    \begin{eqnarray}\label{Pwhlp}
    \left\{ \arraycolsep=1.5pt
       \begin{array}{ll}
        -|x|^{\alpha}\Delta u-\alpha|x|^{\alpha-2}(x\cdot\nabla u)=|x|^\beta w &\quad \mbox{in}\  \mathbb{R}^N\setminus\{0\},\\[2mm]
        -|x|^{\alpha}\Delta w-\alpha|x|^{\alpha-2}(x\cdot\nabla w)
        =\frac{(2^{**}_{\alpha,\beta}-1)C_{N,\alpha,\beta}
        ^{2^{**}_{\alpha,\beta}-2}}
        {|x|^{8+3\beta-4\alpha}(1+|x|^{\alpha-\beta-2})^4}
        u&\quad \mbox{in}\  \mathbb{R}^N\setminus\{0\},
        \end{array}
    \right.
    \end{eqnarray}
    in $u\in\mathcal{D}^{2,2}_{\alpha,\beta}(\mathbb{R}^N)$, due to $(N+\beta)(N+\gamma)=(N+2\alpha-\beta-4)^2$.

    Now we decompose $u$ and $w$ as follows:
    \begin{equation}\label{defvd}
    u(r,\theta)=\sum^{\infty}_{k=0}\sum^{l_k}_{i=1}
    \phi_{k,i}(r)\Psi_{k,i}(\theta),\quad w(r,\theta)=\sum^{\infty}_{k=0}\sum^{l_k}_{i=1}
    \psi_{k,i}(r)\Psi_{k,i}(\theta),
    \end{equation}
    where $r=|x|$, $\theta=x/|x|\in \mathbb{S}^{N-1}$, and
    \begin{equation*}
    \phi_{k,i}(r)=\int_{\mathbb{S}^{N-1}}
    u(r,\theta)\Psi_{k,i}(\theta)\mathrm{d}\theta,\quad \psi_{k,i}(r)=\int_{\mathbb{S}^{N-1}}
    w(r,\theta)\Psi_{k,i}(\theta)\mathrm{d}\theta.
    \end{equation*}
    Here $\Psi_{k,i}(\theta)$ denotes the $k$-th spherical harmonic, i.e., it satisfies
    \begin{equation}\label{deflk}
    -\Delta_{\mathbb{S}^{N-1}}\Psi_{k,i}=\lambda_k \Psi_{k,i},
    \end{equation}
    where $\Delta_{\mathbb{S}^{N-1}}$ is the Laplace-Beltrami operator on $\mathbb{S}^{N-1}$ with the standard metric and  $\lambda_k$ is the $k$-th eigenvalue of $-\Delta_{\mathbb{S}^{N-1}}$. It is well known that \begin{equation}\label{deflklk}
    \lambda_k=k(N-2+k),\quad k=0,1,2,\ldots,
    \end{equation}
    whose multiplicity is
    \[
    l_k:=\frac{(N+2k-2)(N+k-3)!}{(N-2)!k!}
    \]
    (note that $l_0:=1$) and that \[
    \mathrm{Ker}(\Delta_{\mathbb{S}^{N-1}}+\lambda_k)
    =\mathbb{Y}_k(\mathbb{R}^N)|_{\mathbb{S}^{N-1}},
    \]
    where $\mathbb{Y}_k(\mathbb{R}^N)$ is the space of all homogeneous harmonic polynomials of degree $k$ in $\mathbb{R}^N$. It is standard that $\lambda_0=0$ and the corresponding eigenfunction of (\ref{deflk}) is the constant function that is $\Psi_{0,1}=c\in\mathbb{R}\setminus\{0\}$. The second eigenvalue $\lambda_1=N-1$ and the corresponding eigenfunctions of (\ref{deflk}) are $\Psi_{1,i}=x_i/|x|$, $i=1,\ldots,N$.

    It is known that
    \begin{align}\label{Ppwhl2deflklw}
    \Delta (\varphi_{k,i}(r)\Psi_{k,i}(\theta))
    = & \Psi_{k,i}\left(\varphi''_{k,i}+\frac{N-1}{r}\varphi'_{k,i}\right)
    +\frac{\varphi_{k,i}}{r^2}\Delta_{\mathbb{S}^{N-1}}\Psi_{k,i} \nonumber\\
    = & \Psi_{k,i}\left(\varphi''_{k,i}+\frac{N-1}{r}\varphi'_{k,i}
    -\frac{\lambda_k}{r^2}\varphi_{k,i}\right).
    \end{align}
    Furthermore, it is easy to verify that
    \begin{equation*}
    \frac{\partial (\varphi_{k,i}(r)\Psi_{k,i}(\theta))}{\partial x_j}=\varphi'_{k,i}\frac{x_j}{r}\Psi_{k,i}
    +\varphi_{k,i}\frac{\partial\Psi_{k,i}}{\partial \theta_l}\frac{\partial\theta_l}{\partial x_j},\quad \mbox{for all}\quad l=1,\ldots,N-1,
    \end{equation*}
    and
    \begin{equation*}
    \sum^{N}_{j=1}\frac{\partial\theta_l}{\partial x_j}x_j=0,\quad \mbox{for all}\quad l=1,\ldots,N-1,
    \end{equation*}
    hence
    \begin{align}\label{Ppwhl2deflkln}
    x\cdot\nabla (\varphi_{k,i}(r)\Psi_{k,i}(\theta))=
    \sum^{N}_{j=1}x_j\frac{\partial (\varphi_{k,i}(r)\Psi_{k,i}(\theta))}{\partial x_j}=\varphi'_{k,i}r\Psi_{k,i}.
    \end{align}
    Therefore, by standard regularity theory, putting together (\ref{Ppwhl2deflklw}) and (\ref{Ppwhl2deflkln}) into (\ref{Pwhlp}), the functions pair $(u,w)$ is a solution of (\ref{Pwhlp}) if and only if $(\phi_{k,i},\psi_{k,i})\in \mathcal{C}_k\times \mathcal{C}_k$ is a classical solution of the system
    \begin{eqnarray}\label{p2c}
    \left\{ \arraycolsep=1.5pt
       \begin{array}{ll}
        \phi''_{k,i}+\frac{N-1+\alpha}{r}\phi'_{k,i}-
        \frac{\lambda_k}{r^2}\phi_{k,i}
        +\frac{\psi_{k,i}}{r^{\alpha-\beta}}=0 \quad \mbox{in}\  r\in(0,\infty),\\[3mm]
        \psi''_{k,i}+\frac{N-1+\alpha}{r}\psi'_{k,i}
        -\frac{\lambda_k}{r^2}\psi_{k,i}
        +\frac{(2^{**}_{\alpha,\beta}-1)C_{N,\alpha,\beta}
        ^{2^{**}_{\alpha,\beta}-2}}
        {r^{8+3\beta-3\alpha}(1+r^{\alpha-\beta-2})^4}\phi_{k,i}=0 \quad \mbox{in}\  r\in(0,\infty),\\[3mm]
        \phi'_{k,i}(0)=\psi'_{k,i}(0)=0 \quad\mbox{if}\quad k=0,\quad \mbox{and}\quad \phi_{k,i}(0)=\psi_{k,i}(0)=0 \quad\mbox{if}\quad k\geq 1,
        \end{array}
    \right.
    \end{eqnarray}
    for all $i=1,\ldots,l_k$, where
    \[
    \mathcal{C}_k:=\left\{\omega\in C^2((0,\infty))| \int^\infty_0 \left[\omega''(r)+\frac{N+\alpha-1}{r}\omega'(r)-
        \frac{\lambda_k}{r^2}\omega(r)\right]^2
    r^{N+2\alpha-\beta-1}\mathrm{d}r<\infty\right\}.
    \]
    Same as in the proof of Corollary \ref{thmPbcb}, let us make the change of variables as
    \begin{equation*}
    X_{k,i}(s)=r^a\phi_{k,i}(r),\quad Y_{k,i}(s)=q^2 r^a\psi_{k,i}(r),
    \end{equation*}
    with $r=s^q$, where $a=N+\alpha-2$ and $q=\frac{2}{2+\beta-\alpha}$,
    which transforms (\ref{p2c}) into the system
    \begin{eqnarray}\label{p2t}
    \left\{ \arraycolsep=1.5pt
       \begin{array}{ll}
        X''_{k,i}+\frac{M-1}{s}X'_{k,i}
        -\frac{q^2\lambda_k}{s^2}X_{k,i}+Y_{k,i}=0 \quad \mbox{in}\  s\in(0,\infty),\\[2mm]
        Y''_{k,i}+\frac{M-1}{s}Y'_{k,i}
        -\frac{q^2\lambda_k}{s^2}Y_{k,i}
        +\frac{(M+4)(M-2)M(M+2)}{(1+s^2)^4}X_{k,i}=0 \quad \mbox{in}\  s\in(0,\infty),\\[2mm]
        X'_{k,i}(0)=Y'_{k,i}(0)=0 \quad\mbox{if}\quad k=0,\quad \mbox{and}\quad X_{k,i}(0)=Y_{k,i}(0)=0 \quad\mbox{if}\quad k\geq 1,
        \end{array}
    \right.
    \end{eqnarray}
    for all $i=1,\ldots,l_k$, in $(X_{k,i},Y_{k,i})\in \widetilde{\mathcal{C}_k}\times \widetilde{\mathcal{C}_k}$, where
    \[
    \widetilde{\mathcal{C}_k}:=\left\{\omega\in C^2((0,\infty))| \int^\infty_0 \left[\omega''(s)+\frac{M-1}{s}\omega'(s)
    -\frac{q^2\lambda_k}{s^2}\omega(s)\right]^2 s^{M-1} \mathrm{d}s<\infty\right\},
    \]
    and
    \begin{equation*}
    M=\frac{2(N+2\alpha-\beta-4)}{\alpha-\beta-2}>4.
    \end{equation*}
    Here we have used the fact $q^4(2^{**}_{\alpha,\beta}-1)
    C_{N,\alpha,\beta}^{2^{**}_{\alpha,\beta}-2}
    =(M+4)(M-2)M(M+2)$.

    Note that \eqref{p2t} is equivalent to the fourth-order ODE
    \begin{align}\label{rwevpb}
    \left(\Delta_s-\frac{\varpi_k}{s^2}\right)^2X_{k,i}
    & = \left(q^2\lambda_k-\varpi_k\right)
    \left(\frac{2}{s^2}X_{k,i}''
    +\frac{2(M-3)}{s^3}X_{k,i}'
    -\frac{2(M-4)+q^2\lambda_k+\varpi_k}{s^4}X_{k,i}\right)
    \nonumber \\
    &\quad +(\tilde{2}^{**}-1)\Gamma_M(1+s^2)^{-4}X_{k,i},
    \end{align}
    in $s\in(0,\infty)$, $X_{k,i}\in \widetilde{\mathcal{C}_k}$, $i=1,2,\ldots, l_k$. Here $\Delta_s:=\frac{\partial^2}{\partial s^2}+\frac{M-1}{s}
    \frac{\partial}{\partial s}$, $\varpi_k:=k(M-2+k)$, $\tilde{2}^{**}:=\frac{2M}{M-4}$ and
    \begin{align}\label{defgn}
    \Gamma_M:=(M-4)(M-2)M(M+2).
    \end{align}
    Note that $q^2\lambda_k\geq\varpi_k$ for all $k\geq 1$, moreover, $q^2\lambda_k>\varpi_k$ for $\beta>\beta_{\mathrm{FS}}(\alpha)$ and $k\geq 1$, also for $\beta=\beta_{\mathrm{FS}}(\alpha)$ and $k\geq 2$.
    It is easy to verify that when $k=0$, \eqref{rwevpb} admits only one solution $X_0(s)=\frac{1-s^2}{(1+s^2)^{\frac{M-2}{2}}}$ (up to multiplications), see our recent work \cite{DT24} for details. In fact, \cite[Lemma 2.4]{BWW03} states that if $X\in \widetilde{\mathcal{C}_0}$ is a solution of
    \[
    \left[s^{1-M}\frac{\partial}{\partial s}\left(s^{M-1}\frac{\partial}{\partial s}\right)\right]^2X=\nu(1+s^2)^{-4}X \quad \mbox{for}\ \nu>0,
    \]
    with $X(0)=0$, then $X\equiv 0$ (which states that $M$ is an integer, in fact, it also holds for all $M>4$ since the only one step needs to be modified is showing that if $X\not\equiv 0$ then $X$ has only a finite number of positive zeros which requires the conclusions of \cite[Proposition 2]{El77} and \cite[p.273]{Sw92} and they are indeed correct for all $M>4$). Now, let $\tilde{X}_0\not\equiv 0$ be another solution of \eqref{rwevpb} with $k=0$. Then $X_0(0), \tilde{X}_0(0)\neq 0$, for otherwise $X_0$ resp. $\tilde{X}_0$ would vanish identically by \cite[Lemma 2.4]{BWW03}. So we can find $\tau\in\mathbb{R}$ such that $X_0(0)=\tau\tilde{X}_0(0)$. But then $X_0-\tau\tilde{X}_0$ also solves \eqref{rwevpb} with $k=0$ and equals zero at origin. By \cite[Lemma 2.4]{BWW03}, one has $X_0-\tau\tilde{X}_0\equiv 0$, and so $\tilde{X}_0$ is a scalar multiple of $X_0$. When $\beta=\beta_{\mathrm{FS}}(\alpha)$ which implies $q^2\lambda_1=\varpi_1$, then \eqref{rwevpb} with $k=1$ admits one solution $X_1(s)=\frac{s}{(1+s^2)^{\frac{M-2}{2}}}$, in fact, when $M$ is an integer then we can directly obtain the uniqueness of solutions (up to multiplications) by using the standard  stereographic projection as in \cite{Fe02}, furthermore, since \eqref{rwevpb} is an ODE then the uniqueness also holds for all $M>4$ with minor changes.

    {\bf We claim that when $\beta>\beta_{\mathrm{FS}}(\alpha)$ for all $k\geq 1$,  \eqref{rwevpb} does not exist nontrivial solutions, and also when $\beta=\beta_{\mathrm{FS}}(\alpha)$ for all $k\geq 2$}.

    Now, we begin to show this claim when $M$ is an integer. One easily checks the operator identity
    \[
    \left(\Delta_s-\frac{\varpi_k}{s^2}\right)(\cdot)
    =s^k\left[\frac{\partial^2}{\partial s^2}+\frac{M+2k-1}{s}\frac{\partial}{\partial s}\right](s^{-k}\cdot).
    \]
    Define $Y_{k,i}\in C^\infty((0,\infty))$ by $Y_{k,i}(s):=s^{-k}X_{k,i}$, then equation \eqref{rwevpb} can be rewritten as
    \begin{small}
    \begin{align}\label{rwevpbb}
    \left(\frac{\partial^2}{\partial s^2}+\frac{M+2k-1}{s}\frac{\partial}{\partial s}\right)^2 Y_{k,i}
    & = \left(q^2\lambda_k-\varpi_k\right)
    \left[\frac{2}{s^2}Y_{k,i}''
    +\frac{2(M-3)}{s^3}Y_{k,i}'
    -\frac{2(M-4)+q^2\lambda_k+\varpi_k}{s^4}Y_{k,i}\right]
    \nonumber \\
    & \quad +(\tilde{2}^{**}-1)\Gamma_M(1+s^2)^{-4}Y_{k,i}.
    \end{align}
    \end{small}
    So following the work of Bartsch et al. \cite{BWW03}, we deduce that
    \[
    Z_{k,i}\in \mathcal{D}^{2,2}_0(\mathbb{R}^{M+2k}),
    \]
    where $\mathcal{D}^{2,2}_0(\mathbb{R}^{M+2k})$ denotes the completion of $C^\infty_0(\mathbb{R}^{M+2k})$ with respect to the norm
    \[
    \|u\|_{\mathcal{D}^{2,2}_0(\mathbb{R}^{M+2k})}
    =\left(\int_{\mathbb{R}^{M+2k}}|\Delta u|^2 \mathrm{d}y\right)^{\frac{1}{2}},
    \]
    and $Z_{k,i}$ is a weak solution of the equation
    \begin{align}\label{rwevpbbb}
    \Delta^2 Z_{k,i}(y)
    & = \left(q^2\lambda_k-\varpi_k\right)
    \left[\frac{2}{s^2}Z_{k,i}''
    +\frac{2(M-3)}{s^3}Z_{k,i}'
    -\frac{2(M-4)+q^2\lambda_k+\varpi_k}{s^3}Z_{k,i}\right]
    \nonumber \\
    & \quad +(\tilde{2}^{**}-1)\Gamma_M(1+|y|^2)^{-4}Z_{k,i}(y),\quad y\in \mathbb{R}^{M+2k}.
    \end{align}
    Multiplying \eqref{rwevpbbb} by $Z_{k,i}$ and integrating in $\mathbb{R}^{M+2k}$, we have
    \begin{small}\begin{align*}
    \|Z_{k,i}\|^2_{\mathcal{D}^{2,2}_0(\mathbb{R}^{M+2k})}
    & = (\tilde{2}^{**}-1)\Gamma_M\int_{\mathbb{R}^{M+2k}}
    (1+|y|^2)^{-4}|Z_{k,i}|^2 \mathrm{d}y
    -\left(q^2\lambda_k-\varpi_k\right)
    \Bigg\{
    2\int_{\mathbb{R}^{M+2k}}
    \frac{|\nabla Z_{k,i}|^2}{|y|^{2}}
    \mathrm{d}y
    \nonumber \\
    &\quad \quad +[(2(M-4)+q^2\lambda_k+\varpi_k)-2k(M+2k-4)]
    \int_{\mathbb{R}^{M+2k}}\frac{|Z_{k,i}|^2}
    {|y|^{4}}\mathrm{d}y
    \Bigg\}.
    \end{align*}\end{small}
    By the classical Hardy inequality,
    \begin{align*}
    \left(\frac{M+2k-4}{2}\right)^2\int_{\mathbb{R}^{M+2k}}
    \frac{|u|^2}{|y|^{4}}\mathrm{d}y
    \leq \int_{\mathbb{R}^{M+2k}}
    \frac{|\nabla u|^2}{|y|^{2}}\mathrm{d}y, \quad \mbox{for all}\quad u\in C^\infty_0(\mathbb{R}^{M+2k}),
    \end{align*}
    we deduce that
    \begin{align}\label{rwevpbbbcb}
    \|Z_{k,i}\|^2_{\mathcal{D}^{2,2}_0(\mathbb{R}^{M+2k})}
    & \leq (\tilde{2}^{**}-1)\Gamma_M\int_{\mathbb{R}^{M+2k}}
    (1+|y|^2)^{-4}|Z_{k,i}|^2 \mathrm{d}y
    \nonumber\\
    & \quad -\left(q^2\lambda_k-\varpi_k\right)\xi_k
    \int_{\mathbb{R}^{M+2k}}\frac{|Z_{k,i}|^2}
    {|y|^{4}}\mathrm{d}y.
    \end{align}
    Here
    \[
    \xi_k:=\frac{(M+2k-4)^2}{2}+(2(M-4)+q^2\lambda_k+\varpi_k)
    -2k(M+2k-4).
    \]
    Note that
    \begin{align}\label{rwevpbbbcbi}
    \|u\|^2_{\mathcal{D}^{2,2}_0(\mathbb{R}^{M+2k})}
    \geq \Gamma_{M+2k}\int_{\mathbb{R}^{M+2k}}
    \frac{|u(y)|^2}{(1+|y|^2)^{4}} \mathrm{d}y,\quad \mbox{for all}\quad u\in \mathcal{D}^{2,2}_0(\mathbb{R}^{M+2k}),
    \end{align}
    see \cite[(2.10)]{BWW03}, then combining with \eqref{rwevpbbbcb} and \eqref{rwevpbbbcbi} we deduce
    \begin{align*}
    \left[(\tilde{2}^{**}-1)\Gamma_M-\Gamma_{M+2k}\right]
    \int_{\mathbb{R}^{M+2k}}
    (1+|y|^2)^{-4}|Z_{k,i}|^2 \mathrm{d}y
    \geq \left(q^2\lambda_k-\varpi_k\right)\xi_k
    \int_{\mathbb{R}^{M+2k}}\frac{|Z_{k,i}|^2}
    {|y|^{4}}\mathrm{d}y.
    \end{align*}
    Since
    \[
    (\tilde{2}^{**}-1)\Gamma_M\leq \Gamma_{M+2k},\quad \mbox{for all}\quad k\geq 1,
    \]
    and
    \[
    (\tilde{2}^{**}-1)\Gamma_M< \Gamma_{M+2k},\quad \mbox{for all}\quad k\geq 2,
    \]
    where $\Gamma_M$ is defined in \eqref{defgn}, then it holds that
    \begin{align}\label{rwevpbbbcbcf}
    \left(q^2\lambda_k-\varpi_k\right)\xi_k
    \int_{\mathbb{R}^{M+2k}}\frac{|Z_{k,i}|^2}
    {|y|^{4}}\mathrm{d}y\leq 0.
    \end{align}
    Note that
    \begin{align*}
    \xi_k
    & \geq \frac{(M+2k-4)^2}{2}+(2(M-4)+2\varpi_k)
    -2k(M+2k-4) \\
    & = \frac{M^2}{2}+2(M-2)(k-2)+2(M-4)>0,
    \end{align*}
    for all $k\geq 1$,then we conclude from \eqref{rwevpbbbcbcf} that $Z_{k,i}\equiv 0$ for all $k\geq 1$ if $\beta>\beta_{\mathrm{FS}}(\alpha)$, and $Z_{k,i}\equiv 0$ for all $k\geq 2$ if $\beta=\beta_{\mathrm{FS}}(\alpha)$.
    Since \eqref{rwevpb} is ODE, even if $M$ is not an integer we readily see that the above conclusion remains true. Our claim is proved.

    To sum up, let us turn back to (\ref{p2c}) we obtain the solutions that, if $\beta=\beta_{\mathrm{FS}}(\alpha)$ then \eqref{p2c} only admits
    \begin{equation*}
    \phi_0(r)=\frac{r^{2-N-\alpha}(1-r^{2+\beta-\alpha})}
    {(1+r^{2+\beta-\alpha})^{\frac{N-2+\alpha}{\alpha-\beta-2}}}
    =\frac{1-r^{2+\beta-\alpha}}
    {(1+r^{\alpha-\beta-2})^{\frac{N-2+\alpha}{\alpha-\beta-2}}},\quad
    \phi_1(r)=\frac{r^{\frac{2+\beta-\alpha}{2}}}
    {(1+r^{\alpha-\beta-2})^{\frac{N-2+\alpha}{\alpha-\beta-2}}},
    \end{equation*}
    otherwise, \eqref{p2c} only admits $\phi_0$.
    That is, if $\beta=\beta_{\mathrm{FS}}(\alpha)$, the space of solutions of (\ref{Pwhlp}) has dimension $(1+N)$ and is spanned by
    \begin{equation*}
    Z_{0}(x)=\frac{1-|x|^{2+\beta-\alpha}}
    {(1+|x|^{\alpha-\beta-2})^{\frac{N-2+\alpha}{\alpha-\beta-2}}},\quad Z_{i}(x)=\frac{|x|^{\frac{2+\beta-\alpha}{2}}}{(1+|x|^{\alpha-\beta-2})
    ^\frac{N-2+\alpha}{\alpha-\beta-2}}\cdot\frac{x_i}{|x|},\quad i=1,\ldots,N.
    \end{equation*}
    Otherwise the space of solutions of (\ref{Pwhlp}) has only dimension one and is spanned by $Z_0$, and note that $Z_0\thicksim \frac{\partial U_{\lambda}}{\partial \lambda}|_{\lambda=1}$ in this case we say $U$ is non-degenerate. The proof of Theorem \ref{thmpwhl} is now completed.
    \qed

\vskip0.25cm

\subsection{{\bfseries Symmetry breaking phenomenon}}\label{sectsbp}

Now, based on the results of Theorems \ref{thmPbcb} and \ref{thmpwhl}, we are ready to prove our main result. In order to shorten formulas, for each $u\in \mathcal{D}^{2,2}_{\alpha,\beta}(\mathbb{R}^N)$ we denote
    \begin{equation}\label{def:norm}
    \|u\|: =\left(\int_{\mathbb{R}^N}|x|^{-\beta}|\mathrm{div} (|x|^{\alpha}\nabla u)|^2 \mathrm{d}x\right)^{\frac{1}{2}},
    \quad \|u\|_*: =\left(\int_{\mathbb{R}^N}|x|^\gamma|u|^{2^{**}_{\alpha,\beta}} \mathrm{d}x\right)^{\frac{1}{2^{**}_{\alpha,\beta}}}.
    \end{equation}

\vskip0.25cm
\noindent{\bf\em Proof of Theorem \ref{thmmr}.}
    We follow the arguments as those in the proof of \cite[Corollary 1.2]{FS03}. We define the functional $\mathcal{I}$ on $\mathcal{D}^{2,2}_{\alpha,\beta}(\mathbb{R}^N)$ by the right hand side of \eqref{ckns}, i.e.,
    \begin{align}\label{defFu}
    \mathcal{I}(u):=\frac{\|u\|^2}{\|u\|_*^2},\quad u\in \mathcal{D}^{2,2}_{\alpha,\beta}(\mathbb{R}^N)\setminus\{0\},
    \end{align}
    then $\mathcal{I}$ is twice continuously differentiable and
    \begin{small}\begin{align*}
    \langle \mathcal{I}'(u),\varphi\rangle
    & = \frac{2 \int_{\mathbb{R}^N}|x|^{-\beta}\mathrm{div}
    (|x|^{\alpha}\nabla u) \mathrm{div}(|x|^{\alpha}\nabla \varphi)\mathrm{d}x}
    {\|u\|_*^2}
    -\frac{2\|u\|^2}{\|u\|_*^{2^{**}_{\alpha,\beta}+2}}
    \int_{\mathbb{R}^N}|x|^\gamma|u|^{2^{**}_{\alpha,\beta}-2}u\varphi \mathrm{d}x
    \\
    & = \frac{2}{\|u\|_*^2}
    \left(
    \int_{\mathbb{R}^N}|x|^{-\beta}\mathrm{div}(|x|^{\alpha}\nabla u) \mathrm{div}(|x|^{\alpha}\nabla \varphi)\mathrm{d}x
    -\frac{\|u\|^2}{\|u\|_*^{2^{**}_{\alpha,\beta}}}
    \int_{\mathbb{R}^N}|x|^\gamma|u|^{2^{**}_{\alpha,\beta}-2}u\varphi \mathrm{d}x\right).
    \end{align*}\end{small}
    Moreover, a short computation leads to
    \begin{align*}
    \langle \mathcal{I}''(u)\varphi_1,\varphi_2\rangle
    & = \frac{2}{\|u\|_*^2}\int_{\mathbb{R}^N}|x|^{-\beta}
    \mathrm{div}(|x|^{\alpha}\nabla \varphi_1) \mathrm{div}(|x|^{\alpha}\nabla \varphi_2)\mathrm{d}x
    \\
    & \quad -\frac{2(2^{**}_{\alpha,\beta}-1)\|u\|^2}
    {\|u\|_*^{2^{**}_{\alpha,\beta}+2}}
    \int_{\mathbb{R}^N}|x|^\gamma|u|^{2^{**}_{\alpha,\beta}-2}\varphi_1 \varphi_2 \mathrm{d}x
    \\
    & \quad -\frac{4}{\|u\|_*^{2^{**}_{\alpha,\beta}+2}}
    \int_{\mathbb{R}^N}|x|^{-\beta}
    \mathrm{div}(|x|^{\alpha}\nabla u) \mathrm{div}(|x|^{\alpha}\nabla \varphi_1)\mathrm{d}x
    \int_{\mathbb{R}^N}|x|^\gamma
    |u|^{2^{**}_{\alpha,\beta}-2}u\varphi_2 \mathrm{d}x
    \\
    & \quad -\frac{4}{\|u\|_*^{2^{**}_{\alpha,\beta}+2}}
    \int_{\mathbb{R}^N}|x|^{-\beta}
    \mathrm{div}(|x|^{\alpha}\nabla u) \mathrm{div}(|x|^{\alpha}\nabla \varphi_2)\mathrm{d}x
    \int_{\mathbb{R}^N}|x|^\gamma
    |u|^{2^{**}_{\alpha,\beta}-2}u\varphi_1 \mathrm{d}x
    \\
    & \quad + \frac{2(2^{**}_{\alpha,\beta}-2)\|u\|^2}
    {\|u\|_*^{2\cdot2^{**}_{\alpha,\beta}+2}}
    \int_{\mathbb{R}^N}|x|^\gamma |u|^{2^{**}_{\alpha,\beta}-2}u \varphi_1 \mathrm{d}x
    \int_{\mathbb{R}^N}|x|^\gamma |u|^{2^{**}_{\alpha,\beta}-2}u \varphi_2 \mathrm{d}x.
    \end{align*}
    Define also the energy functional of Euler-Lagrange equation \eqref{Pwh} as
    \[
    \mathcal{J}(u):=\frac{1}{2}\|u\|^2
    -\frac{1}{2^{**}_{\alpha,\beta}}\|u\|^{2^{**}_{\alpha,\beta}}_*, \quad u\in \mathcal{D}^{2,2}_{\alpha,\beta}(\mathbb{R}^N).
    \]
    Note that
    \[
    \langle \mathcal{J}'(u),\varphi\rangle
    =\int_{\mathbb{R}^N}|x|^{-\beta}
    \mathrm{div}(|x|^{\alpha}\nabla u) \mathrm{div}(|x|^{\alpha}\nabla \varphi)\mathrm{d}x
    -\int_{\mathbb{R}^N}|x|^\gamma |u|^{2^{**}_{\alpha,\beta}-2}u\varphi \mathrm{d}x,
    \]
    and
    \begin{align*}
    \langle \mathcal{J}''(u)\varphi_1,\varphi_2\rangle
    & = \int_{\mathbb{R}^N}|x|^{-\beta}
    \mathrm{div}(|x|^{\alpha}\nabla \varphi_1) \mathrm{div}(|x|^{\alpha}\nabla \varphi_2)\mathrm{d}x
    -(2^{**}_{\alpha,\beta}-1)
    \int_{\mathbb{R}^N}|x|^\gamma |u|^{2^{**}_{\alpha,\beta}-3}u\varphi_1 \varphi_2 \mathrm{d}x.
    \end{align*}
    From Theorem \ref{thmpwh}, we know that $U$ given in \eqref{defula} is a positive critical of $\mathcal{J}$, thus $\langle \mathcal{J}'(U),\varphi\rangle=0$ for all $\varphi\in \mathcal{D}^{2,2}_{\alpha,\beta}(\mathbb{R}^N)$ and $\|U\|^2=\|U\|_*^{2^{**}_{\alpha,\beta}}$.
    Then we obtain for all $\varphi_1,\varphi_2\in \mathcal{D}^{2,2}_{\alpha,\beta}(\mathbb{R}^N)$,
    \begin{align*}
    \langle \mathcal{I}'(U),\varphi_1\rangle
    = \frac{2}{\|U\|_*^2}
    \langle \mathcal{J}'(U),\varphi_1\rangle
    =0,
    \end{align*}
    and
    \begin{align*}
    \langle \mathcal{I}''(U)\varphi_1,\varphi_2\rangle
    & = \frac{2}{\|U\|_*^2}\langle \mathcal{J}''(U)\varphi_1,\varphi_2\rangle
    \\
    & \quad -\frac{4}{\|U\|_*^{2^{**}_{\alpha,\beta}+2}}
    \int_{\mathbb{R}^N}|x|^{-\beta}
    \mathrm{div}(|x|^{\alpha}\nabla U) \mathrm{div}(|x|^{\alpha}\nabla \varphi_1)\mathrm{d}x
    \int_{\mathbb{R}^N}|x|^\gamma
    U^{2^{**}_{\alpha,\beta}-1}\varphi_2 \mathrm{d}x
    \\
    & \quad -\frac{4}{\|U\|_*^{2^{**}_{\alpha,\beta}+2}}
    \int_{\mathbb{R}^N}|x|^{-\beta}
    \mathrm{div}(|x|^{\alpha}\nabla U) \mathrm{div}(|x|^{\alpha}\nabla \varphi_2)\mathrm{d}x
    \int_{\mathbb{R}^N}|x|^\gamma
    U^{2^{**}_{\alpha,\beta}-1}\varphi_1 \mathrm{d}x
    \\
    & \quad + \frac{2(2^{**}_{\alpha,\beta}-2)}{\|U\|_*^{2^{**}_{\alpha,\beta}+2}}
    \int_{\mathbb{R}^N}|x|^\gamma U^{2^{**}_{\alpha,\beta}-1} \varphi_1 \mathrm{d}x
    \int_{\mathbb{R}^N}|x|^\gamma U^{2^{**}_{\alpha,\beta}-1} \varphi_2 \mathrm{d}x.
    \end{align*}
    From the proof of Theorem \ref{thmpwhl}, we know that $X_1(s)=\frac{s}{(1+s^2)^{\frac{M-2}{2}}}$ with $M=\frac{2(N+2\alpha-\beta-4)}{\alpha-\beta-2}>4$ is a solution of the following system
    \begin{eqnarray*}
    \left\{ \arraycolsep=1.5pt
       \begin{array}{ll}
        X_1''+\frac{M-1}{s}X_1'-\frac{M-1}{s^2}X_1+Y=0 \quad \mbox{in}\ s\in(0,\infty),\\[2mm]
        Y''+\frac{M-1}{s}Y'-\frac{M-1}{s^2}Y
        +\frac{(M+4)(M-2)M(M+2)}{(1+s^2)^4}X_1=0 \quad \mbox{in}\ s\in(0,\infty).
        \end{array}
    \right.
    \end{eqnarray*}
    Then let
    $Z_{i}(x):=|x|^{2-N-\alpha}X_1(|x|^{\frac{2+\beta-\alpha}{2}})
    \frac{x_i}{|x|}$ for some $i\in \{1,2,\ldots,N\}$ same as in \eqref{defaezki}, we deduce that
    \begin{align*}
    \langle\mathcal{J}''(U)Z_{i},Z_{i}\rangle
    & = \frac{\omega_{N-1}}{N}\left(\frac{\alpha-\beta-2}{2}\right)^3
    \Bigg[
    \int^\infty_0
    \left(X_1''+\frac{M-1}{s}X_1'-\frac{\chi}{s^2}X_1\right)^2
    s^{M-1}\mathrm{d}s
    \\
    &\quad -\int^\infty_0
    \left(X_1''+\frac{M-1}{s}X_1'-\frac{M-1}{s^2}X_1\right)^2
    s^{M-1}\mathrm{d}s
    \Bigg]
    \\
    & = \frac{\omega_{N-1}}{N}\left(\frac{\alpha-\beta-2}{2}\right)^3
    [\chi-(M-1)]
    \bigg\{
    2\int^\infty_0(X_1'(s))^2 s^{M-4}\mathrm{d}s
    \\
    & \quad\quad +[(2M-5)+\chi]\int^\infty_0(X_1(s))^2 s^{M-5}\mathrm{d}s
    \bigg\}
    < 0,
    \end{align*}
    where
    \[
    \chi=\left(\frac{2}{2+\beta-\alpha}\right)^2(N-1)<M-1,
    \]
    due to $\alpha>0$ and $\frac{N-4}{N-2}\alpha-4<\beta<\beta_{\mathrm{FS}}(\alpha)$. Therefore, we have
    \[
    \langle \mathcal{I}''(U)Z_{i},Z_{i}\rangle
    =\frac{2}{\|U\|_*^2}\langle \mathcal{J}''(U)Z_{i},Z_{i}\rangle<0,
    \]
    due to $\int_{\mathbb{R}^N}|x|^\gamma U^{2^{**}_{\alpha,\beta}-1} Z_{i}\mathrm{d}x=0$. Consequently, $\mathcal{S}$ is strictly small than $\mathcal{I}(U)=\mathcal{S}_r$, then no minimizers for $\mathcal{S}$ are radial symmetry. Now, the proof of Theorem \ref{thmmr} is completed.
    \qed

\vskip0.25cm

\section{{\bfseries Partial symmetry result}}\label{sectps}

In this section, we will show a symmetry result for the high order (CKN) inequality \eqref{ckn2n} when $2^{**}_{\alpha,\beta}=2^{**}$ and $2-N<\alpha<0$, and give the proof of Theorem \ref{thm2ps} then we show the stability of its extremal functions as in Theorem \ref{thmafsr}.

Let us recall a crucial Rellich-Sobolev type inequality with explicit form of minimizers which was established by Dan et al. \cite{DMY20} as the following.

    \begin{theorem}\label{thmrsi} \cite[Theorem 1.6]{DMY20} Let $N\geq 5$ and $0<\mu<N-4$. For all $v\in \mathcal{D}^{2,2}_{0}(\mathbb{R}^N)$,
    \begin{align*}
    & \int_{\mathbb{R}^N}|\Delta v|^2 \mathrm{d}x
    -C_{\mu,1}\int_{\mathbb{R}^N}\frac{|\nabla v|^2}{|x|^2} \mathrm{d}x
    +C_{\mu,2}\int_{\mathbb{R}^N}\frac{v^2}{|x|^4} \mathrm{d}x
    \geq \left(1-\frac{\mu}{N-4}\right)^{\frac{4N-4}{N}}\mathcal{S}_0
    \left(\int_{\mathbb{R}^N}|v|^{2^{**}} \mathrm{d}x\right)^\frac{2}{2^{**}},
    \end{align*}
    where $\mathcal{S}_0$ is given as in \eqref{cssi} and
    \begin{align*}
    C_{\mu,1}:& =\frac{N^2-4N+8}{2(N-4)^2}\mu[2(N-4)-\mu];
    \\
    C_{\mu,2}:& =\frac{N^2}{16(N-4)^2}\mu^2[2(N-4)-\mu]^2
    -\frac{N-2}{2}\mu[2(N-4)-\mu].
    \end{align*}
    Moreover, equality holds if and only if
    \begin{align*}
    v(x)=c|x|^{-\frac{\mu}{2}}
    \left(\lambda^2+|x|^{2(1-\frac{\mu}{N-4})}
    \right)^{-\frac{N-4}{2}},
    \end{align*}
    for all $c\in\mathbb{R}$ and $\lambda>0$.
    \end{theorem}

    Let us make a brief comment about the proof of Theorem \ref{thmrsi}. The key step is the change of variable
    \begin{align*}
    v(x)=|x|^{-\frac{\mu}{2}}w(|x|^{-\frac{\mu}{N-4}}x).
    \end{align*}
    Then by using standard spherical decomposition which needs lots of careful calculations, it leads to
    \begin{align*}
    \int_{\mathbb{R}^N}|\Delta v|^2 \mathrm{d}x
    -C_{\mu,1}\int_{\mathbb{R}^N}\frac{|\nabla v|^2}{|x|^2} \mathrm{d}x
    +C_{\mu,2}\int_{\mathbb{R}^N}\frac{v^2}{|x|^4} \mathrm{d}x\geq \left(1-\frac{\mu}{N-4}\right)^3\int_{\mathbb{R}^N}|\Delta w|^2 \mathrm{d}x,
    \end{align*}
    and the equality holds if and only if $v$ is radial symmetry. Furthermore,
    \begin{align*}
    \int_{\mathbb{R}^N}|v|^{2^{**}} \mathrm{d}x
    =\left(1-\frac{\mu}{N-4}\right)^{-1}\int_{\mathbb{R}^N}|w|^{2^{**}} \mathrm{d}x.
    \end{align*}
    Then by using the classical second-order Sobolev inequality given as in \eqref{cssio} to complete the proof. Now, we are ready to prove our partial symmetry result Theorem \ref{thm2ps}.

\vskip0.25cm
\noindent{\bf\em Proof of Theorem \ref{thm2ps}.}
    For each $u\in \mathcal{D}^{2,2}_{\alpha,\frac{N-4}{N-2}\alpha-4}(\mathbb{R}^N)$, let us make the change
    \[
    u(x)=|x|^{\eta}v(x) \quad\mbox{with}\quad
    \eta=-2-\frac{N}{2(N-2)}\alpha.
    \]
    A direct calculation indicates
    \begin{align}\label{psle}
    \int_{\mathbb{R}^N}
    |x|^{2\cdot2^{**}+\frac{N^2}{(N-2)(N-4)}\alpha}|u|^{2^{**}} \mathrm{d}x
    =\int_{\mathbb{R}^N}
    |v|^{2^{**}} \mathrm{d}x,
    \end{align}
    due to $2\cdot2^{**}+\frac{N^2}{(N-2)(N-4)}\alpha+\eta\cdot 2^{**}=0$, and
    \begin{align}\label{psny}
    & \int_{\mathbb{R}^N}|x|^{4-\frac{N-4}{N-2}\alpha}|\mathrm{div} (|x|^{\alpha}\nabla u)|^2 \mathrm{d}x
    \nonumber\\
    &=\int_{\mathbb{R}^N}|x|^{4-\frac{N-4}{N-2}\alpha}
    [|x|^{\eta+\alpha}\Delta v +(2\eta+\alpha)|x|^{\alpha+\eta-2}(x\cdot \nabla v)
    +\eta(N+\alpha+\eta-2)|x|^{\alpha+\eta-2}v]^2 \mathrm{d}x
    \nonumber\\
    & = \int_{\mathbb{R}^N}
    [\Delta v +(2\eta+\alpha)|x|^{-2}(x\cdot \nabla v)
    +\eta(N+\alpha+\eta-2)|x|^{-2}v]^2 \mathrm{d}x
    \nonumber\\
    & = \int_{\mathbb{R}^N}|\Delta v|^2 \mathrm{d}x
    +\eta^2(N+\alpha+\eta-2)^2\int_{\mathbb{R}^N}|x|^{-4}v^2\mathrm{d}x
    +(2\eta+\alpha)^2\int_{\mathbb{R}^N}|x|^{-4}(x\cdot \nabla v)^2\mathrm{d}x
    \nonumber\\
    &\quad +2(2\eta+\alpha)\int_{\mathbb{R}^N}|x|^{-2}(x\cdot \nabla v)\Delta v\mathrm{d}x
    +2(N+\alpha+\eta-2)\int_{\mathbb{R}^N}|x|^{-2}v\Delta v\mathrm{d}x
    \nonumber\\
    &\quad +2(2\eta+\alpha)\eta(N+\alpha+\eta-2)
    \int_{\mathbb{R}^N}|x|^{-4}v(x\cdot \nabla v)\mathrm{d}x,
    \end{align}
    due to $2(\alpha+\eta)+4-\frac{N-4}{N-2}\alpha=0$. From \eqref{eqi}, we obtain
    \begin{align}\label{psd}
    (N-4)\int_{\mathbb{R}^N}
    |x|^{-2}|\nabla v|^2\mathrm{d}x
    & =
    2\int_{\mathbb{R}^N}
    (x\cdot\nabla v)\mathrm{div}(|x|^{-2}\nabla v)
    \mathrm{d}x
    \nonumber\\
    & = 2\int_{\mathbb{R}^N}
    |x|^{-2}(x\cdot\nabla v)\Delta v
    \mathrm{d}x-4\int_{\mathbb{R}^N}|x|^{-4}(x\cdot\nabla v)^2
    \mathrm{d}x,
    \end{align}
    and by Green's formula
    \begin{align}\label{psp}
    \int_{\mathbb{R}^N}|x|^{-2}v\Delta v\mathrm{d}x
    & =-\int_{\mathbb{R}^N}\nabla(|x|^{-2}v)\cdot\nabla v\mathrm{d}x
    \nonumber\\
    & = -\int_{\mathbb{R}^N}|x|^{-2}|\nabla v|^2\mathrm{d}x
    +2\int_{\mathbb{R}^N}|x|^{-4}v(x\cdot\nabla v)
    \mathrm{d}x,
    \end{align}
    and by Divergence formula
    \begin{align}\label{psdf}
    \int_{\mathbb{R}^N}|x|^{-4}v(x\cdot \nabla v)\mathrm{d}x
    & =-\int_{\mathbb{R}^N}v\mathrm{div}(|x|^{-4}xv)\mathrm{d}x
    \nonumber\\
    & = -(N-4)\int_{\mathbb{R}^N}|x|^{-4}v^2\mathrm{d}x
    -\int_{\mathbb{R}^N}|x|^{-4}v(x\cdot \nabla v)\mathrm{d}x
    \nonumber\\
    & = -\frac{N-4}{2}\int_{\mathbb{R}^N}|x|^{-4}v^2\mathrm{d}x.
    \end{align}
    Therefore, from \eqref{psd}, \eqref{psp} and \eqref{psdf} we deduce
    \begin{align*}
    & (2\eta+\alpha)^2\int_{\mathbb{R}^N}|x|^{-4}(x\cdot \nabla v)^2\mathrm{d}x+2(2\eta+\alpha)\int_{\mathbb{R}^N}|x|^{-2}(x\cdot \nabla v)\Delta v\mathrm{d}x
    \\
    &\quad+2(N+\alpha+\eta-2)\int_{\mathbb{R}^N}|x|^{-2}v\Delta v\mathrm{d}x
    +2(2\eta+\alpha)\eta(N+\alpha+\eta-2)
    \int_{\mathbb{R}^N}|x|^{-4}v(x\cdot \nabla v)\mathrm{d}x
    \\
    &= -(N-4)\eta(N+\alpha+\eta-2)(2\eta+\alpha+2)
    \int_{\mathbb{R}^N}|x|^{-4}v^2\mathrm{d}x
    \\
    &\quad+ (2\eta+\alpha)(2\eta+\alpha+4)\int_{\mathbb{R}^N}|x|^{-4}(x\cdot\nabla v)^2 \mathrm{d}x
    \\
    &\quad +[(2\eta+\alpha)(N-4)-2\eta(N+\alpha+\eta-2)]
    \int_{\mathbb{R}^N}|x|^{-2}|\nabla v|^2\mathrm{d}x
    \\
    & \geq -(N-4)\eta(N+\alpha+\eta-2)(2\eta+\alpha+2)
    \int_{\mathbb{R}^N}|x|^{-4}v^2\mathrm{d}x
    \\
    &\quad +[(2\eta+\alpha)(N+2\eta+\alpha)-2\eta(N+\alpha+\eta-2)]
    \int_{\mathbb{R}^N}|x|^{-2}|\nabla v|^2\mathrm{d}x,
    \end{align*}
    due to $(x\cdot\nabla v)^2\leq |x|^2|\nabla v|^2$ and $(2\eta+\alpha)(2\eta+\alpha+4)=\frac{4\alpha}{N-2}
    \left(2+\frac{\alpha}{N-2}\right)<0$ (note that the assumption $\alpha<0$ plays a crucial role), furthermore, the equality holds if and only if $v$ is radial.
    Thus from \eqref{psny} we have
    \begin{align}\label{psnb}
    & \int_{\mathbb{R}^N}|x|^{4-\frac{N-4}{N-2}\alpha}|\mathrm{div} (|x|^{\alpha}\nabla u)|^2 \mathrm{d}x
    \geq \int_{\mathbb{R}^N}|\Delta v|^2 \mathrm{d}x
    \nonumber\\
    &\quad+ [(2\eta+\alpha)(N+2\eta+\alpha)-2\eta(N+\alpha+\eta-2)]
    \int_{\mathbb{R}^N}|x|^{-2}|\nabla v|^2\mathrm{d}x
    \nonumber\\
    &\quad+ \left[\eta^2(N+\alpha+\eta-2)^2
    -(N-4)\eta(N+\alpha+\eta-2)(2\eta+\alpha+2)\right]
    \int_{\mathbb{R}^N}|x|^{-4}v^2\mathrm{d}x,
    \end{align}
    and the equality holds if and only if $v$ is radial (so does $u$ due to $u(x)=|x|^\eta v(x)$). Note that $\eta=-2-\frac{N}{2(N-2)}\alpha$, then let us make the change
    \[
    \alpha=\frac{2-N}{N-4}\mu,
    \]
    which implies $0<\mu<N-4$ due to $2-N<\alpha<0$,  we have
    \begin{align*}
    (2\eta+\alpha)(N+2\eta+\alpha)-2\eta(N+\alpha+\eta-2)
    & =\frac{[2(N-2)+\alpha]}\alpha(N^2-4N+8){2(N-2)^2}
    \\
    & = -\frac{N^2-4N+8}{2(N-4)^2}\mu[2(N-4)-\mu]
    =-C_{\mu,1},
    \end{align*}
    and
    \begin{align*}
    & \eta^2(N+\alpha+\eta-2)^2-(N-4)\eta(N+\alpha+\eta-2)(2\eta+\alpha+2)
    \\
    &= \frac{N^2}{16(N-4)^2}\mu^2[2(N-4)-\mu]^2
    -\frac{N-2}{2}\mu[2(N-4)-\mu]
    =C_{\mu,2},
    \end{align*}
    where $C_{\mu,1}$ and $C_{\mu,2}$ are given as in Theorem \ref{thmrsi}. Then combining with \eqref{psnb} and Theorem \ref{thmrsi}, we have
    \begin{align}\label{psbb}
    \int_{\mathbb{R}^N}|x|^{4-\frac{N-4}{N-2}\alpha}|\mathrm{div} (|x|^{\alpha}\nabla u)|^2 \mathrm{d}x
    & \geq \int_{\mathbb{R}^N}|\Delta v|^2 \mathrm{d}x
    - C_{\mu,1}\int_{\mathbb{R}^N}|x|^{-2}|\nabla v|^2\mathrm{d}x
    + C_{\mu,2}\int_{\mathbb{R}^N}|x|^{-4}v^2\mathrm{d}x
    \nonumber\\
    & \geq \left(1-\frac{\mu}{N-4}\right)^{4-\frac{4}{N}}\mathcal{S}_0
    \left(\int_{\mathbb{R}^N}|v|^{2^{**}} \mathrm{d}x\right)^\frac{2}{2^{**}},
    \end{align}
    the first equality holds only if $v$ is radial, and the second equality holds only if $v(x)=c|x|^{-\frac{\mu}{2}}
    \left(\lambda^2+|x|^{2(1-\frac{\mu}{N-4})}
    \right)^{-\frac{N-4}{2}}$ for all $c\in\mathbb{R}$ and $\lambda>0$. Therefore from \eqref{psle} we deduce
    \begin{align*}
    \int_{\mathbb{R}^N}|x|^{4-\frac{N-4}{N-2}\alpha}|\mathrm{div} (|x|^{\alpha}\nabla u)|^2 \mathrm{d}x
    \geq \left(1+\frac{\alpha}{N-2}\right)^{4-\frac{4}{N}}\mathcal{S}_0
    \left(\int_{\mathbb{R}^N}
    |x|^{2\cdot2^{**}+\frac{N^2}{(N-2)(N-4)}\alpha}|u|^{2^{**}} \mathrm{d}x\right)^{\frac{2}{2^{**}}},
    \end{align*}
    and the equality holds if and only if
    \[
    u(x)=c|x|^\eta|x|^{-\frac{\mu}{2}}
    \left(\lambda^2+|x|^{2(1-\frac{\mu}{N-4})}
    \right)^{-\frac{N-4}{2}}
    =c|x|^{-2-\frac{2\alpha}{N-2}}
    (\lambda^2+|x|^{2+\frac{2\alpha}{N-2}})
    ^{-\frac{N-4}{2}},
    \]
    that is, $u(x)=c\lambda^{\frac{N}{2}(1+\frac{\alpha}{N-2})}U(\lambda x)$ where
    \begin{align*}
    U(x)=C_{N,\alpha,\frac{N-4}{N-2}\alpha-4}|x|^{-2-\frac{2\alpha}{N-2}}
    (1+|x|^{2+\frac{2\alpha}{N-2}})
    ^{-\frac{N-4}{2}}
    \end{align*}
    is as in Theorem \ref{thmpwh} replacing $\beta$ by $\frac{N-4}{N-2}\alpha-4$.
    Now the proof of Theorem \ref{thm2ps} is completed.
    \qed

\vskip0.25cm

\subsection{{\bfseries Stability of extremal functions}}\label{sectsbr}
    Now, we are going to show the stability of extremal functions for inequality \eqref{ckn2ps} and give the proof of Theorem \ref{thmafsr}, by using spectral analysis as in \cite{BWW03,BE91}. In order to shorten formulas, we write $\beta=\frac{N-4}{N-2}\alpha-4$, $\gamma=2\cdot2^{**}+\frac{N^2}{(N-2)(N-4)}\alpha$ as the original problem, and $\mathcal{S}_\alpha:=\left(1+\frac{\alpha}{N-2}\right)
    ^{4-\frac{4}{N}}\mathcal{S}_0$. Also, the norms $\|\cdot\|$ and $\|\cdot\|_*$ are defined as in \eqref{def:norm}.

    Let us consider the following eigenvalue problem
    \begin{equation}\label{Pwhlep}
    \mathrm{div}(|x|^{\alpha}\nabla(|x|^{-\beta}
    \mathrm{div}(|x|^\alpha\nabla u)))
    =\nu |x|^\gamma U^{2^{**}-2}u \quad \mbox{in}\  \mathbb{R}^N\setminus\{0\},
    \end{equation}
    $u\in \mathcal{D}^{2,2}_{\alpha,\beta}(\mathbb{R}^N)$. It is easy to verify that $\mathcal{D}^{2,2}_{\alpha,\beta}(\mathbb{R}^N)$ embeds compactly into the weighted space $L^2(\mathbb{R}^N,|x|^\gamma U^{2^{**}-2}\mathrm{d}x)$ (see \cite[Theorem 3.2]{DT23-rs} with minor changes), which indicates the eigenvalues of problem \eqref{Pwhlep} are discrete. Then it is well known that the first eigenvalue of problem \eqref{Pwhlep} can be defined as
    \begin{equation}\label{deffev1}
    \nu_1:=\inf_{u\in \mathcal{D}^{2,2}_{\alpha,\beta}(\mathbb{R}^N) \setminus\{0\}}
    \frac{\|u\|^2}
    {\int_{\mathbb{R}^N}|x|^\gamma U^{2^{**}-2} u^2 \mathrm{d}x}.
    \end{equation}
    Moreover, for any $k\in\mathbb{N}^+$ the eigenvalues can be characterized as follows:
    \begin{equation}\label{deffevk}
    \nu_{k+1}:=\inf_{u\in \mathbb{P}_{k+1}\setminus\{0\}}
    \frac{\|u\|^2}
    {\int_{\mathbb{R}^N}|x|^\gamma U^{2^{**}-2} u^2 \mathrm{d}x},
    \end{equation}
    where
    \begin{align*}
    \mathbb{P}_{k+1}: & =\bigg\{u\in \mathcal{D}^{2,2}_{\alpha,\beta}(\mathbb{R}^N): \int_{\mathbb{R}^N}|x|^{-\beta}
    \mathrm{div}(|x|^\alpha\nabla u)\mathrm{div}(|x|^\alpha\nabla e_{i,j})\mathrm{d}x=0,
    \\
    & \quad \quad \mbox{for all}\quad i=1,\ldots,k,\ j=1,\ldots,h_i\bigg\},
    \end{align*}
    and $e_{i,j}$ are the corresponding eigenfunctions to $\nu_i$ with $h_i$ multiplicity. Taking $U$ as a test function in \eqref{deffev1} we know $\nu_1\leq 1$, moreover, by H\"{o}lder's inequality and embedding inequality \eqref{ckn2ps} we have
    \begin{align*}
    \int_{\mathbb{R}^N}|x|^\gamma U^{2^{**}-2} u^2 \mathrm{d}x
    & \leq \left(\int_{\mathbb{R}^N}|x|^\gamma U^{2^{**}}\mathrm{d}x\right)^{\frac{2^{**}-2}{2^{**}}}
    \left(\int_{\mathbb{R}^N}|x|^\gamma |u|^{2^{**}}\mathrm{d}x\right)^{\frac{2}{2^{**}}}
    =\mathcal{S}_\alpha
    \|u\|^2_*
    \leq \|u\|^2,
    \end{align*}
    and equalities hold if and only if $u=cU$ for all $c\in\mathbb{R}$, thus $\nu_1=1$ and the corresponding eigenfunction is $U$ (up to nonzero multiplications). Furthermore, we notice that $U$ minimizes the functional
    \begin{align}\label{deffe}
    u\mapsto \Phi(u)=\frac{1}{2}\|u\|^2-\frac{1}{2^{**}}
    \|u\|^{2^{**}}_*,
    \end{align}
    on the Nehari manifold
    \begin{align*}
    \mathcal{N}:=\left\{u\in \mathcal{D}^{2,2}_{\alpha,\beta}(\mathbb{R}^N) \backslash\{0\}: \|u\|^2=\|u\|^{2^{**}}_*\right\}.
    \end{align*}
    Indeed, for $u\in \mathcal{N}$ we have by inequality \eqref{ckn2ps} that
    \begin{align*}
    \Phi(v)
    & =\left(\frac{1}{2}-\frac{1}{2^{**}}\right)
    \|u\|^{2^{**}}_*
    = \left(\frac{1}{2}-\frac{1}{2^{**}}\right)\left(\frac{\|u\|}
    {\|u\|_*}\right)^{\frac{2\cdot 2^{**}}{2^{**}-2}}
    \\
    & \geq\left(\frac{1}{2}-\frac{1}{2^{**}}\right)
    \mathcal{S}_\alpha^{\frac{2^{**}}{2^{**}-2}}
    =  \left(\frac{1}{2}-\frac{1}{2^{**}}\right)\left(\frac{\|U\|}
    {\|U\|_*}\right)^{\frac{2\cdot 2^{**}}{2^{**}-2}}
    = \Phi(U).
    \end{align*}
    As a consequence, the second derivative $\Phi''(U)$ given by
    \begin{align*}
    \langle\Phi''(U)\phi,\varphi\rangle
    =\int_{\mathbb{R}^N}|x|^{-\beta}
    \mathrm{div}(|x|^\alpha\nabla \phi)\mathrm{div}(|x|^\alpha\nabla \varphi)\mathrm{d}x
    -(2^{**}-1)\int_{\mathbb{R}^N}
    |x|^\gamma U^{2^{**}-2}\phi\varphi \mathrm{d}x
    \end{align*}
    is nonnegative quadratic form when restricted to the tangent space $T_{U}\mathcal{N}$, then we have
    \[
    \|u\|^2
    \geq (2^{**}-1)\int_{\mathbb{R}^N}
    |x|^\gamma U^{2^{**}-2}u^2\mathrm{d}x,
    \quad \mbox{for all}\quad u\in T_{U}\mathcal{N}.
    \]
    Since $T_{U}\mathcal{N}$ has codimension one, we infer that $\nu_2\geq 2^{**}-1$. Moreover, since $\frac{\partial U_\lambda}{\partial \lambda}|_{\lambda=1}$ is a solution of \eqref{Pwhlep} with $\nu=2^{**}-1$ which indicates $\nu_2\leq 2^{**}-1$, then we conclude that $\nu_2= 2^{**}-1$. Then the non-degeneracy of $U$ given as in Theorem \ref{thmpwhl} shows that $\frac{\partial U_\lambda}{\partial \lambda}|_{\lambda=1}$ is the only (up to nonzero multiplications) eigenfunction of $\nu_2= 2^{**}-1$. Then from the definition of eigenvalues, we deduce that there is a constant $\nu_3>2^{**}-1$ such that
    \[
    \|u\|^2
    \geq \nu_3\int_{\mathbb{R}^N}
    |x|^\gamma U^{2^{**}-2}u^2\mathrm{d}x,\quad \mbox{for all}\ u\bot \mathrm{Span}\left\{U,\frac{\partial U_\lambda}{\partial \lambda}\Big|_{\lambda=1}\right\}.
    \]
    Here $u\bot \mathrm{Span}\left\{U,\frac{\partial U_\lambda}{\partial \lambda}|_{\lambda=1}\right\}$ means
    \[
    \int_{\mathbb{R}^N}|x|^{-\beta}
    \mathrm{div}(|x|^\alpha\nabla u)\mathrm{div}(|x|^\alpha\nabla \phi)\mathrm{d}x=0,\quad \mbox{for all}\ \phi\in \mathrm{Span}\left\{U,\frac{\partial U_\lambda}{\partial \lambda}\Big|_{\lambda=1}\right\}.
    \]

    Since $\mathcal{M}_\alpha=\{cU_\lambda: c\in\mathbb{R},\ \lambda>0\}$ is two-dimensional manifold embedded in $\mathcal{D}^{2,2}_{\alpha,\beta}(\mathbb{R}^N)$, and the tangential space at $(c,\lambda)$ is given by
    \begin{align*}
    T_{cU_{\lambda}}\mathcal{M}_\alpha=\mathrm{Span}
    \left\{U_{\lambda},\ \frac{\partial U_{\lambda}}{\partial \lambda}\right\},
    \end{align*}
    then the above spectral analysis indicates that
    \begin{align}\label{czkj}
    \nu_3\int_{\mathbb{R}^N}
    |x|^\gamma
    U_{\lambda}^{2^{**}-2}u^2\mathrm{d}x \leq \|u\|^2, \quad \mbox{for all}\ u\bot T_{cU_{\lambda}}\mathcal{M}_\alpha,
    \end{align}
    where $\nu_3>\nu_2=2^{**}-1$ is independent of $\lambda$. The main ingredient in the proof of Theorem \ref{thmafsr} is contained in the lemma below, where the behavior near extremals set $\mathcal{M}_\alpha$ is studied.

    \begin{lemma}\label{lemma:rtnm2b}
    Suppose $N\geq 5$ and $2-N<\alpha<0$. Then for any sequence $\{u_n\}\subset \mathcal{D}^{2,2}_{\alpha,\beta}(\mathbb{R}^N)
    \backslash \mathcal{M}_\alpha$ satisfying $\inf\limits_n\|u_n\|>0$ and $\inf\limits_{w\in \mathcal{M}_\alpha}\|u_n-w\|\to 0$, it holds that
    \begin{equation}\label{rtnmb}
    \liminf\limits_{n\to\infty}\frac{\|u_n\|^2
    -\mathcal{S}_\alpha\|u\|^2_*}
    {\inf\limits_{w\in \mathcal{M}_\alpha}\|u_n-w\|^2}\geq 1-\frac{\nu_2}{\nu_3}.
    \end{equation}
    \end{lemma}

    \begin{proof}
    The proof now is quite standard, and we can refer to our recent work \cite{DT23-rs} (proof of Lemma 4.1) for details with minor changes, so here we omit it.
    \end{proof}

    Now, we are ready to prove the stability of extremal functions for inequality \eqref{ckn2ps}.

\vskip0.25cm

    \noindent{\bf \em Proof of Theorem \ref{thmafsr}.} We argue by contradiction. In fact, if the theorem is false then there exists a sequence $\{u_n\}\subset \mathcal{D}^{2,2}_{\alpha,\beta}(\mathbb{R}^N)
    \setminus \mathcal{M}_\alpha$ satisfying $\inf\limits_n\|u_n\|>0$ and $\inf\limits_{w\in \mathcal{M}_\alpha}\|u_n-w\|^2\to 0$, such that
    \begin{equation*}
    \frac{\|u_n\|^2
    -\mathcal{S}_\alpha\|u\|_*^2}
    {\inf\limits_{w\in \mathcal{M}_\alpha}\|u_n-w\|^2}
    \to 0,\quad \mbox{as}\quad n\to \infty.
    \end{equation*}
    By homogeneity, we can assume that $\|u_n\|=1$, and after selecting a subsequence we can assume that $\inf\limits_{w\in \mathcal{M}_\alpha}\|u_n-w\|\to \xi\in[0,1]$ since $\inf\limits_{w\in \mathcal{M}_\alpha}\|u_n-w\|\leq \|u_n\|$. If $\xi=0$, then we deduce a contradiction by Lemma \ref{lemma:rtnm2b}.

    The other possibility only is that $\xi>0$, that is
    \[
    \inf_{w\in \mathcal{M}_\alpha}\|u_n-w\|\to \xi>0\quad \mbox{as}\quad n\to \infty,
    \]
    then we must have
    \begin{equation}\label{wbsi}
    \|u_n\|^2
    -\mathcal{S}_\alpha\|u\|^2_*\to 0,\quad \|u_n\|=1.
    \end{equation}
    Since $\mathcal{S}_\alpha<\mathcal{S}_0$ for all $2-N<\alpha<0$, by taking the same arguments as those in \cite{WW00}, we can deduce that there is a sequence $\{\lambda_n\}\subset\mathbb{R}^+$ such that $\{(u_n)_{\lambda_n}\}$ contains a convergent subsequence which will lead to a contradiction.
    In fact, as stated previous about the work of Dan et al. \cite{DMY20} and also the proof of Theorem \ref{thmafsr}, we can make the change
    \begin{equation}\label{p2txyb}
    u_n(x)=|x|^{-2-\frac{2\alpha}{N-2}}v_n(|x|^{\frac{\alpha}{N-2}}x),
    \end{equation}
    then
    \begin{align*}
    0 \leq \|\Delta v_n\|^2_{L^2(\mathbb{R}^N)}
    -\mathcal{S}_0\|v_n\|^2_{L^{2^{**}}(\mathbb{R}^N)}
    \leq
    \left(1+\frac{\alpha}{N-2}\right)^{-3}
    \left(\|u_n\|^2-\mathcal{S}_\alpha\|u_n\|^2_*\right)
    \to 0.
    \end{align*}
    Note that the change \eqref{p2txyb} implies $v_n$ does not invariant under translation, then by Lions' concentration-compactness principle (see \cite[Theorem \uppercase\expandafter{\romannumeral 2}.4]{Li85-1}), there is a sequence $\{\tau_n\}\subset\mathbb{R}^+$ such that
    \begin{equation*}
    \tau_n^{\frac{N-4}{2}}v_n(\tau_n x)\to W\quad \mbox{in}\quad \mathcal{D}^{2,2}_0(\mathbb{R}^N)\quad \mbox{as}\quad n\to \infty,
    \end{equation*}
    where $W(x)=c(d+|x|^2)^{-\frac{N-4}{2}}$ for some $c\neq 0$ and $d>0$, thus
    \begin{equation*}
    \lambda_n^{\frac{N}{2}(1+\frac{\alpha}{N-2})}u_n(\lambda_n x)\to U_*\quad \mbox{in}\quad \mathcal{D}^{2,2}_{\alpha,\beta}(\mathbb{R}^N)\quad \mbox{as}\quad n\to \infty,
    \end{equation*}
    for some $U_*\in\mathcal{M}_\alpha$, where $\lambda_n=\tau_n^{\frac{N-2}{N-2+\alpha}}$, which implies
    \begin{equation*}
    \inf\limits_{w\in \mathcal{M}_\alpha}\|u_n-w\|
    =\inf\limits_{w\in \mathcal{M}_\alpha}\left\|\lambda_n^{\frac{N}{2}(1+\frac{\alpha}{N-2})}
    u_n(\lambda_n x)-w\right\|\to 0 \quad \mbox{as}\quad n\to \infty,
    \end{equation*}
    this is a contradiction. Now, the proof of Theorem \ref{thmafsr} is completed.
    \qed

\appendix

\section{\bfseries A new second-order (CKN) inequality}\label{sectpls}

Based on the works of \cite{GG22} and \cite{Li86}, we will establish a new second-order (CKN) inequality shown as in \eqref{ckn2n} which is equivalent to the classical one \eqref{ckn2Y}. As stated in the introduction, we only need to give an equivalent form. Let us recall the high order (CKN) inequality \eqref{cknh}. Note that, when $j=0$, $m=1$, $p=2$ and $b=a+1$ satisfying $a<\frac{N-2}{2}$, we obtain the weighted Hardy inequality:
    \begin{equation}\label{cknwhi}
    \int_{\mathbb{R}^N}|x|^{-2a-2}|u|^2 \mathrm{d}x
    \leq C\int_{\mathbb{R}^N}|x|^{-2a}|\nabla u|^{2} \mathrm{d}x,
    \quad \mbox{for all}\quad u\in C^\infty_0(\mathbb{R}^N\setminus\{0\}).
    \end{equation}
    In fact, the classical Hardy inequality states that
    \[
    \int_{\mathbb{R}^N}\frac{|u|^2} {|x|^{2}}\mathrm{d}x
    \leq \left(\frac{N-2}{2}\right)^2\int_{\mathbb{R}^N}|\nabla u|^{2} \mathrm{d}x,
    \quad \mbox{for all}\quad u\in C^\infty_0(\mathbb{R}^N),
    \]
    then by using the change $u(x)=|x|^{-a}v(x)$ we have
    \begin{small}\begin{equation}\label{cknwhib}
    \int_{\mathbb{R}^N}|x|^{-2a-2}|v|^2 \mathrm{d}x
    \leq \left(\frac{N-2a-2}{2}\right)^2\int_{\mathbb{R}^N}|x|^{-2a}|\nabla v|^{2} \mathrm{d}x,
    \quad \mbox{for all}\quad v\in C^\infty_0(\mathbb{R}^N\setminus\{0\}),
    \end{equation}\end{small}
    Furthermore, when $j=1$, $m=2$, $p=2$ and $b=a+1$ satisfying $a<\frac{N-2}{2}$, we also obtain the weighted Hardy-Rellich inequality:
    \begin{equation}\label{cknhri}
    \int_{\mathbb{R}^N}|x|^{-2a-2}|\nabla u|^2 \mathrm{d}x
    \leq C\int_{\mathbb{R}^N}|x|^{-2a}|\Delta u|^{2} \mathrm{d}x,
    \quad \mbox{for all}\quad u\in C^\infty_0(\mathbb{R}^N\setminus\{0\}).
    \end{equation}

    Now, let us give the equivalent form. Note that the assumption \eqref{cknc} implies $-N<\beta<N-4+2\alpha$, here we will give a more general conclusion.

\begin{proposition}\label{propneq}
Assume that $N\geq 5$, $-N<\beta<N-4+2\alpha$. Then there is a constant $\mathfrak{C}=\mathfrak{C}(N,\alpha,\beta)>1$ such that
for all $u\in C^\infty_0(\mathbb{R}^N\setminus\{0\})$,
\begin{align}\label{neq}
\frac{1}{\mathfrak{C}}\int_{\mathbb{R}^N}|x|^{-\beta}|\mathrm{div} (|x|^{\alpha}\nabla u)|^2 \mathrm{d}x
\leq \int_{\mathbb{R}^N}|x|^{2\alpha-\beta}|\Delta u|^2 \mathrm{d}x
\leq \mathfrak{C} \int_{\mathbb{R}^N}|x|^{-\beta}|\mathrm{div} (|x|^{\alpha}\nabla u)|^2 \mathrm{d}x.
\end{align}
\end{proposition}

\begin{proof}
We follow the arguments as those in \cite[Section 2]{GG22} and also \cite[Proposition A.1]{DT23-f}. Note that
\[
\mathrm{div} (|x|^{\alpha}\nabla u)=|x|^{\alpha}\Delta u+\alpha|x|^{\alpha-2}x\cdot\nabla u,
\]
then we obtain
\begin{align}\label{neqe}
\int_{\mathbb{R}^N}|x|^{-\beta}|\mathrm{div} (|x|^{\alpha}\nabla u)|^2\mathrm{d}x
= & \int_{\mathbb{R}^N}|x|^{2\alpha-\beta}|\Delta u|^2\mathrm{d}x
+ 2\alpha \int_{\mathbb{R}^N}|x|^{2\alpha-2-\beta}\Delta u(x\cdot\nabla u)\mathrm{d}x
\nonumber \\
& + \alpha^2 \int_{\mathbb{R}^N}|x|^{2\alpha-4-\beta}(x\cdot\nabla u)^2\mathrm{d}x.
\end{align}
By H\"{o}lder's inequality we have
\begin{align*}
\left|\int_{\mathbb{R}^N}|x|^{2\alpha-2-\beta}\Delta u(x\cdot\nabla u)\mathrm{d}x\right|
& \leq \int_{\mathbb{R}^N}|x|^{2\alpha-1-\beta}|\Delta u||\nabla u|\mathrm{d}x
\\
& \leq \left(\int_{\mathbb{R}^N}|x|^{2\alpha-\beta}|\Delta u|^2\mathrm{d}x\right)^{\frac{1}{2}}
\left(\int_{\mathbb{R}^N}|x|^{2\alpha-\beta-2}|\nabla u|^2\mathrm{d}x\right)^{\frac{1}{2}},
\end{align*}
thus the Young's inequality implies
\begin{align*}
\left|2\alpha \int_{\mathbb{R}^N}|x|^{2\alpha-2-\beta}\Delta u(x\cdot\nabla u)\mathrm{d}x\right|
\leq |\alpha|\left(\int_{\mathbb{R}^N}|x|^{2\alpha-\beta}|\Delta u|^2\mathrm{d}x+\int_{\mathbb{R}^N}|x|^{2\alpha-\beta-2}|\nabla u|^2\mathrm{d}x\right).
\end{align*}
Furthermore,
\[
\int_{\mathbb{R}^N}|x|^{2\alpha-4-\beta}(x\cdot\nabla u)^2\mathrm{d}x
\leq \int_{\mathbb{R}^N}|x|^{2\alpha-2-\beta}|\nabla u|^2\mathrm{d}x.
\]
Since $-\frac{2\alpha-\beta}{2}<\frac{N-2}{2}$, by using the weighted Hardy-Rellich inequality \eqref{cknhri} we obtain
\[
\int_{\mathbb{R}^N}|x|^{2\alpha-\beta-2}|\nabla u|^2\mathrm{d}x
\leq D \int_{\mathbb{R}^N}|x|^{2\alpha-\beta}|\Delta u|^2\mathrm{d}x,
\]
for some $D>0$ independent of $u$, then from \eqref{neqe} we deduce the left inequality in \eqref{neq} with $\mathfrak{C}=1+(1+D)|\alpha|+D\alpha^2$.

Then we show the right inequality in \eqref{neq}. Consider
\begin{align}\label{neqer}
\Upsilon:=-|x|^{-\frac{\beta}{2}}\mathrm{div} (|x|^{\alpha}\nabla u),
\end{align}
then we obtain
\begin{align}\label{neqer1}
\int_{\mathbb{R}^N}|x|^{-\beta}|\mathrm{div} (|x|^{\alpha}\nabla u)|^2\mathrm{d}x
=\int_{\mathbb{R}^N}|\Upsilon|^2\mathrm{d}x.
\end{align}
It follows from \eqref{neqer} that
\begin{align*}
|x|^{\alpha-\frac{\beta}{2}}\Delta u=-\Upsilon-\alpha|x|^{\alpha-\frac{\beta}{2}-2}(x\cdot \nabla u),
\end{align*}
thus
\begin{align}\label{neqer2}
\int_{\mathbb{R}^N}|x|^{2\alpha-\beta}|\Delta u|^2\mathrm{d}x
& =\int_{\mathbb{R}^N}|\Upsilon|^2\mathrm{d}x
+2\alpha\int_{\mathbb{R}^N}|x|^{\alpha-\frac{\beta}{2}-2}(x\cdot \nabla u)\Upsilon\mathrm{d}x
\nonumber\\
&\quad +\alpha^2\int_{\mathbb{R}^N} |x|^{2\alpha-\beta-4}(x\cdot \nabla u)^2\mathrm{d}x.
\end{align}
Taking $\phi=|x|^{\alpha-\beta-2} u$ as a test function to \eqref{neqer}, we obtain
\[
\int_{\mathbb{R}^N}|x|^{\alpha}\nabla u\cdot \nabla \phi\mathrm{d}x
=\int_{\mathbb{R}^N}|x|^{\frac{\beta}{2}}\Upsilon \phi\mathrm{d}x,
\]
that is,
\[
\int_{\mathbb{R}^N}|x|^{2\alpha-\beta-2}|\nabla u|^2\mathrm{d}x +(\alpha-\beta-2)\int_{\mathbb{R}^N}|x|^{2\alpha-\beta-4}
\left(x\cdot \nabla \left(\frac{|u|^2}{2}\right)\right)\mathrm{d}x
=\int_{\mathbb{R}^N} |x|^{\alpha-\frac{\beta}{2}-2}u\Upsilon \mathrm{d}x.
\]
Therefore, by divergence theorem we have
\[
\int_{\mathbb{R}^N} |x|^{2\alpha-\beta-2}|\nabla u|^2\mathrm{d}x -\frac{\alpha-\beta-2}{2}
\int_{\mathbb{R}^N} |u|^2\mathrm{div}(|x|^{2\alpha-\beta-4}x)\mathrm{d}x
=\int_{\mathbb{R}^N} |x|^{\alpha-\frac{\beta}{2}-2}u\Upsilon \mathrm{d}x,
\]
then
\begin{align}\label{neqere}
\int_{\mathbb{R}^N} |x|^{\alpha-\frac{\beta}{2}-2}u\Upsilon\mathrm{d}x
& =\int_{\mathbb{R}^N} |x|^{2\alpha-\beta-2}|\nabla u|^2\mathrm{d}x
\nonumber \\
&\quad -\frac{(\alpha-\beta-2)(N+2\alpha-\beta-4)}{2}
\int_{\mathbb{R}^N}|x|^{2\alpha-\beta-4}|u|^2\mathrm{d}x.
\end{align}
If $\alpha-\beta-2\leq 0$, from \eqref{neqere} we have
\[
\int_{\mathbb{R}^N} |x|^{2\alpha-\beta-2}|\nabla u|^2\mathrm{d}x
\leq \int_{\mathbb{R}^N} |x|^{\alpha-\frac{\beta}{2}-2}u\Upsilon\mathrm{d}x.
\]
Otherwise, if $\alpha-\beta-2>0$, then since $-\frac{2\alpha-\beta-2}{2}<\frac{N-2}{2}$, \eqref{cknwhib} indicates
\begin{align*}
\frac{(\alpha-\beta-2)(N+2\alpha-\beta-4)}{2}
\int_{\mathbb{R}^N}|x|^{2\alpha-\beta-4}|u|^2\mathrm{d}x
\leq \frac{2(\alpha-\beta-2)}{N+2\alpha-\beta-4}
\int_{\mathbb{R}^N} |x|^{2\alpha-\beta-2}|\nabla u|^2\mathrm{d}x,
\end{align*}
then from \eqref{neqere} we have
\[
\left(1-\frac{2(\alpha-\beta-2)}{N+2\alpha-\beta-4}\right)
\int_{\mathbb{R}^N} |x|^{2\alpha-\beta-2}|\nabla u|^2\mathrm{d}x
\leq \int_{\mathbb{R}^N} |x|^{\alpha-\frac{\beta}{2}-2}u\Upsilon\mathrm{d}x.
\]
Note that the assumption $-N<\beta<N-4+2\alpha$ implies $\frac{2(\alpha-\beta-2)}{N+2\alpha-\beta-4}<1$, thus by H\"{o}lder's inequality and from \eqref{cknwhi} (or \eqref{cknwhib}) we can always obtain
\begin{align*}
\int_{\mathbb{R}^N} |x|^{2\alpha-\beta-2}|\nabla u|^2\mathrm{d}x
& \leq
C_1\left(\int_{\mathbb{R}^N} |x|^{2\alpha-\beta-4}|u|^2\mathrm{d}x\right)^{\frac{1}{2}}
\left(\int_{\mathbb{R}^N} |\Upsilon|^2\mathrm{d}x\right)^{\frac{1}{2}}
\\
& \leq C_2\left(\int_{\mathbb{R}^N} |x|^{2\alpha-\beta-2}|\nabla u|^2\mathrm{d}x\right)^{\frac{1}{2}}
\left(\int_{\mathbb{R}^N} |\Upsilon|^2\mathrm{d}x\right)^{\frac{1}{2}}
\\
& \leq E\int_{\mathbb{R}^N} |\Upsilon|^2\mathrm{d}x,
\end{align*}
for some $E>0$ independent of $u$.
By Young's inequality, it follows from \eqref{neqer1}-\eqref{neqer2} that
\begin{align*}
\int_{\mathbb{R}^N}|x|^{2\alpha-\beta}|\Delta u|^2\mathrm{d}x
& \leq \int_{\mathbb{R}^N}|\Upsilon|^2\mathrm{d}x
+|\alpha|\left(\int_{\mathbb{R}^N}|x|^{2\alpha-\beta-2}|\nabla u|^2\mathrm{d}x
+\int_{\mathbb{R}^N}|\Upsilon|^2\mathrm{d}x
\right)
\\
& \quad +\alpha^2\int_{\mathbb{R}^N} |x|^{2\alpha-\beta-2}|\nabla u|^2\mathrm{d}x
\\
& = (1+|\alpha|)\int_{\mathbb{R}^N}|\Upsilon|^2 \mathrm{d}x +(|\alpha|+\alpha^2)\int_{\mathbb{R}^N} |x|^{2\alpha-\beta-2}|\nabla u|^2\mathrm{d}x
\\
& \leq \left[1+|\alpha|+E(|\alpha|+\alpha^2)\right]
\int_{\mathbb{R}^N}|x|^{-\beta}|\mathrm{div} (|x|^{\alpha}\nabla u)|^2\mathrm{d}x.
\end{align*}
Therefore, we obtain the right inequality in \eqref{neq} with $\mathfrak{C}=1+|\alpha|+E(|\alpha|+\alpha^2)$.

Note that the constant $\mathfrak{C}$ in \eqref{neq} can be chosen as
\[
\mathfrak{C}=\max\{1+|\alpha|(1+D)+D\alpha^2, 1+|\alpha|(1+E)+E\alpha^2\}.
\]
Now the proof is completed.
\end{proof}

\vskip0.25cm

\noindent{\bfseries Acknowledgements:}

The research has been supported by National Natural Science Foundation of China (No. 12371121).

\noindent{\bfseries Declarations}:

{\bf Conflict of Interest}: The authors declare that they have no conflict of interest.

\noindent
{\bf Data Availability:}

Date sharing is not applicable to this article as no new data were created analyzed in this study.


    \end{document}